\documentclass[11pt,a4paper]{article}

\usepackage[english,ngerman]{babel}
\usepackage[T1]{fontenc}
\usepackage[ansinew]{inputenc}
\usepackage{here} 
\usepackage{amsfonts}
\usepackage{amsmath}
\usepackage{amsthm}
\usepackage{amssymb}
\usepackage{mathrsfs}

\usepackage{vector}
\usepackage{easymat}		
\usepackage{easyeqn}		
\usepackage{graphicx}
\usepackage{bbm}
\usepackage{stmaryrd}		
\usepackage{epsfig}
\usepackage{lmodern}

\def\f{\mathfrak }
\def\b{\mathbb }
\def\c{\mathcal }

\def\phi{\varphi }
\def\epsilon{\varepsilon}
\newcommand{\p}{\quad}

\newcommand{\be}{\begin{equation}}
\newcommand{\ee}{\end{equation}}
\newcommand{\beo}{\begin{equation*}}
\newcommand{\eeo}{\end{equation*}}
\newcommand{\beao}{\begin{eqnarray*}}
\newcommand{\eeao}{\end{eqnarray*}}

\newcommand{\row}[2]{#1_1,\ldots,#1_{#2}} 
\def\<#1>{\langle#1\rangle} 

\theoremstyle{plain}
\newtheorem{trm}{Theorem}[section]

\newtheorem{pro}[trm]{Proposition}
\newtheorem{lmm}[trm]{Lemma}
\newtheorem{cor}[trm]{Corollary}

\theoremstyle{definition}
\newtheorem{dfn}[trm]{Definition}
\newtheorem{rem}[trm]{Remark}

\numberwithin{equation}{section}

\begin{document}

\selectlanguage{english}
\title{Limit theorems for radial random walks on Euclidean spaces 
of high dimensions}
  \author{Waldemar Grundmann\\
  e-mail:  waldemar.grundmann@math.tu-dortmund.de} 
\maketitle
\begin{abstract}
Let $\nu\in M^1([0,\infty[)$ be a fixed probability measure. For each dimension $p\in \mathbb{N}$, let $(X_n^{p})_{n\geq1}$ be i.i.d. $\mathbb{R}^p$-valued random variables with radially symmetric distributions and radial distribution $\nu$. 
We investigate the distribution of the Euclidean length of $S_n^{p}:=X_1^{p}+\ldots + X_n^{p}$ for large parameters $n$ and $p$. Depending on the growth of the dimension $p=p_n$ we derive by the method of moments two complementary CLT's for the functional $\|S_n^{p}\|_2$ with normal limits, namely for $n/p_n \to \infty$ and $n/p_n \to 0$. Moreover, we present a CLT for the case $n/p_n \to c\in ]0,\infty[$. Thereby we derive explicit formulas and asymptotic results for moments of radial distributed random variables on $\b R^p$.  

All limit theorems are considered also for orthogonal invariant random walks on the space $\b M_{p,q}(\b R)$ of $p\times q$ matrices instead of $\b R^p$ for $p\to \infty$ and some fixed dimension $q$.
 
\end{abstract}

\newenvironment{rcases}{%
  \left.\renewcommand*\lbrace.%
  \begin{cases}}%
{\end{cases}\right\rbrace}

\newcommand{\norm}[1] {\left\| #1 \right\|}
\newcommand{\dr}[1]{\begin{rotate}{90} #1 \end{rotate}}

\DeclareRobustCommand{\skalarw}{\langle\rangle{0pt}{}}
\DeclareRobustCommand{\binomw}{\genfrac ( ){0pt}{}}

\section{Introduction}\label{eukRad_P1}
The results in this paper are motivated by the following problem:
Let $\nu\in \c M^1([0,\infty[)$ be a fixed probability measure. Then for each dimension $p\in \b N$ there is a unique rotation invariant probability measure $\nu_p\in \c M^1(\b R^p)$ with $\varphi_p(\nu_p)=\nu$, where $\varphi_p(x):=\left\|x\right\|_2$ is the norm mapping. For each $p\in \b N$ consider i.i.d. $\b R^p$-valued random variables $X_k^p$, $k\in \b N$, with law $\nu_p$ as well as the associated radial random walks 
\beo \Big( S_n^p:=\sum_{k=1}^{n}X_k^p \Big)_{n\geq 0} \eeo on $\b R^p$. 
We are interested in finding central limit theorems for the $[0,\infty[$-valued random variables $\left\|S_n^p\right\|_2$ for $n,p\to \infty$ coupled in a suitable way. In this paper we derive the following two associated central limit theorems under disjoint growth conditions for $p=p_n$.
\begin{trm}\label{mainResult0}
Assume that $\nu \in \c M^1([0,\infty[)$ admits finite moments $r_k(\nu):=\int_0^{\infty}x^k d\nu(x)<\infty$ for $k\leq 4$. Let $(p_n)_n$ be a sequence of dimensions with $\lim_{n\to\infty}p_n=\infty$. 
\begin{itemize} 
\item[(1)]
If $\lim\limits_{n\to\infty}\frac{n}{p_n}=\infty$, then 
\beo \frac{\sqrt{p_n}}{n} \bigl( \left\|S_n^{p_n}\right\|_2^2-nr_2(\nu) \bigl) \eeo
tends in distribution for $n\to \infty$ to the normal distribution $\c N(0, 2r_2(\nu)^2)$.
\item[(2)] 
If $\lim\limits_{n\to\infty}\frac{n}{p_n}=c\in [0,\infty[$, then 
\beo \frac{1}{\sqrt{n}} \bigl( \left\|S_n^{p_n}\right\|_2^2-nr_2(\nu) \bigl) \eeo
tends  in distribution for $n\to \infty$ to the normal distribution $\c N(0,r_4(\nu)-(1-2c)r_2(\nu)^2)$.
\end{itemize}
\end{trm}
Parts of this theorem were derived in \cite{voit2} by using completely different methods. More precisely, CLT's above were proven for sequences $(p_n)_n$ with some strong restriction. The first CLT with the restriction $n/p_n^3\to \infty$, was identified by M. Voit as an obvious consequence of Berry-Esseen estimates on $\b R^p$ with explicit constants depending on the dimension $p$, which are due to Bentkus and Götze \cite{bentkus,BenGöt}. The proof of the second CLT with the restriction  $n^2/p_n \to 0$, was derived in \cite{voit2} as a consequence of asymptotic properties of  so called Bessel convolutions (for a survey about the Bessel convolutions we recommend \cite{roesler}).

With the approach used in \cite{voit2} one is not able to get rid of the strong conditions on the growth of $p=p_n$. 
In particular, the mixed case $p_n=c\cdot n$ for some constant $c$, which builds a bridge between the CLT's with $n<<p_n$ and $n>>p_n$ was stated there as an open problem.  

Other associated limit theorems as laws of large numbers and large deviation principle were studied in \cite{röslervoit}. For example, there was proven that 
\beo \frac{1}{n}\left\|S_n^{p_n}\right\|_2^2 \longrightarrow \int_0^{\infty}x^2d\nu(x) \p \b P \text{ - a.s. }\eeo  under the condition that $p_n$ grows fast enough.

Theorem \ref{mainResult0} will appear as special case of an extension which concerns a matrix-valued version. We consider the following geometric situation: For $p,\ q\in \b N$ we will denote by $\b M_{p,q}$ the space of $p\times q$-matrices over the field of real numbers $\b R$. 
Let further $\b H_q$ be the space of symmetric $q\times q$-matrices. 
Moreover, we will denote by $\Pi_q$ the cone of positive semidefinite $q\times q$ matrices in $\b H_q$
We regard $\b M_{p,q}$ as a real vector space of dimension $pq$, equipped with the Euclidean scalar product $\langle x,y\rangle:=tr(x'y)$ and norm $\left\|x\right\|=\sqrt{tr(x'x)}$ where $x'$ is the transpose of $x$ and $tr$ is the trace in $\b M_q:=\b M_{q,q}$.
In the square case $p=q$, $\left\|\cdot \right\|$ is just the Frobenius norm. The orthogonal group $\b O_p$ acts on $\b M_{p,q}$ by left multiplication,
\be\label{OperationOnMpq} \b O_p\times \b M_{p,q}\to \b M_{p,q}, \p (A,x)\mapsto Ax. \ee
By uniqueness of the polar decomposition, two matrices $x, \ y\in \b M_{p,q}$ belong to the same $\b O_p$-orbit if and only if $x'x=y'y$. Thus the space $\b M_{p,q}^{\b O_p}$ of $\b O_p$-orbits in $\b M_{p,q}$ is naturally parameterized by the cone $\Pi_q$ via the map
\beo  x^{\b O_p} \mapsto \sqrt{x'x}=:\left|x\right|, \quad \b M_{p,q}^{\b O_p} \to \Pi_q,  \eeo 
where for $r\in \Pi_q$, the matrix $\sqrt{r}\in \Pi_q$ denotes the unique positive semidefinite square root of $r$. According to this, the map 
\beo \varphi_p: \b M_{p,q}\to \Pi_q, \quad x\mapsto \sqrt{x'x} \eeo will be regarded as the canonical projection $\b M_{p,q}\to \b M_{p,q}^{\b O_p}$. 

In the case $q=1$ we have $\b M_{p,1}\cong \b R^p$, $\b H_1=\b R$, $\Pi_1=[0,\infty[$ and $\varphi_p$ is the usual norm mapping $\left\|\cdot \right\|_2: \b R^p \to [0,\infty[$. Let us now fix a parameter $q\in \b N$. By taking images of measures, $\varphi_p$ induces a Banach space isomorphism between the space $\c M_b^{\b O_p}(\b M_{p,q})$ of all bounded radial (i.e. $\b O_p$ invariant) Borel measures on $\b M_{p,q}$ and the space $\c M_b(\Pi_q)$ of bounded Borel measures on the cone $\Pi_q$. In particular, for each measure $\nu\in \c M^1(\Pi_q)$ and parameter $p$ there is a unique radial probability measure $\nu_p:=\nu_{p,q}\in \c M^1(\b M_{p,q})$ with $\varphi_p(\nu_p)=\nu$. 

Let $\nu\in \c M^1(\Pi_q)$ be a fixed probability measure and $q\in \b N$. As  in the case $q=1$, we now consider for each ``dimension'' $p\in \b N$ the associated radial measures $\nu_p$ on $\b M_{p,q}$ and the radial random walks $(S_n^p:=\sum_{k=1}^n X_k^p)_{n\geq 0}$, i.e. $X_k^p$, $k\in \b N$ are independent $\nu_p$-distributed random variables. 

With this notations, we shall derive the following generalization of Theorem \ref{mainResult0}:
\begin{trm}\label{mainResult1}
Assume that $\nu \in \c M^1(\Pi_q)$ with $\int_{\Pi_q} \left\| s \right\|^4 d\nu(s) <\infty$. Let $(p_n)_{n\in\b N}$ be a sequence of dimensions with $\lim_{n\to\infty}p_n=\infty$.
\begin{itemize}
\item [(1)] If $\lim\limits_{n\to\infty}\frac{n}{p_n}=\infty$, then the $\Pi_q$-valued random variable
\beo \frac{\sqrt{p_n}}{n} \Bigl( \varphi_{p_n}^2 ( S_n^{p_n} )-n \int_{\Pi_q} s^2 d\nu(s)\Bigr) \eeo
tends in distribution to some normal distribution $\c N(0,T^2(\nu))$ on the vector space $\b M_q$ with some covariance matrix $T^2(\nu)$, wich will be described in Theorem \ref{mainResult} precisely. 
\item [(2)]  If $\lim\limits_{n\to\infty}\frac{n}{p_n}=c\in [0,\infty[$, then the $\Pi_q$-valued random variable
\beo \frac{1}{\sqrt{n}}\Bigl( \varphi_{p_n}^2 ( S_n^{p_n} )-n \int_{\Pi_q} s^2 d\nu(s) \Bigr) \eeo
tends in distribution to the normal distribution $\c N(0,\Sigma^2(\nu)+c T^2(\nu))$ on the vector space $\b M_q$ where $\Sigma^2(\nu)$ is the covariance matrix of the $\Pi_q$-valued random variable $\varphi_{p_n}^2(X_{1}^{p_n})$. Note that $\Sigma^2(\nu)$ depends only on $\nu$ and is independent of $p_n$.  
\end{itemize}
\end{trm}

We shall derive Theorem \ref{mainResult1} in this higher rank setting in Section \ref{eukRad_P4}. The proof will rely on asymptotic results for moment functions of so called radial distributed random variables on $\b M_{p,q}$ for $p\to \infty$ as well as on some identities for matrix variate normal distributions.

The organization of the paper is a follows: In section \ref{eukRad_P2}, some preliminaries for the proof of the main result \ref{mainResult1} are presented. More precisely, in Subsection \ref{KHProducts}, after recalling some basic facts about relevant matrix algebra we derive a generalization of so called permutation equivalence property for Kronecker products. In \ref{PermutationMultiset} we 
generalize the multinomial theorem for non commutative operations. In Subsection \ref{BesselFunctions}, background on Bessel functions on the cone $\Pi_q$ is provided. Subsections \ref{PolynomialsOnMpq}-\ref{MVNDistribution} are devoted to the study on the moments of radial measures and of matrix variate normal distributions respectively. In Section \ref{eukRad_P4} our main result is formulated and proved.

\section{Preliminaries}\label{eukRad_P2}
\subsection{Kronecker and Hadamard products}\label{KHProducts}
In this section we collect some known facts  about Kronecker and Hadamard products. The material is taken from \cite{HornJohnson}.

Let $\otimes$ denotes the \textit{Kronecker product} over the field of real numbers $\b R$, that is, $\otimes$ is an operation on two matrices of arbitrary size over $\b R$ resulting in a block matrix. It gives the matrix of the tensor product with respect to a standard choice of basis. With that the Kronecker product of $A=\left[a_{ij}\right]\in \b M_{m,n}$ and  $B=\left[b_{ij}\right]\in \b M_{p, q}$ is the block matrix 
\beo A\otimes B := \left[a_{ij}B\right] \in \b M_{mp, nq}. \eeo
The Kronecker product is bilinear and associative but not commutative. However, $A\otimes B$ and $B\otimes A$ are \textit{permutation equivalent}, meaning that there exist permutation matrices $P$ and $Q$ such that
\be\label{PermutationEquivalent} A\otimes B=P\cdot(B\otimes A)\cdot Q. \ee
If $A$ and $B$ are square matrices, then $A\otimes B$ and $B\otimes A$ are even \textit{permutation similar}, meaning that we can take $P=Q'$. If $A$, $B$, $C$ and $D$ are matrices of such size that one can form the matrix products $A\cdot C$ and $B\cdot D$, then 
\be\label{MixedProductProperty}(A\otimes B)\cdot (C\otimes D)= A\cdot C\otimes B\cdot D. \ee 
This is called the \textit{mixed-product property}, because it mixes the ordinary matrix product and the Kronecker product. If two matrices $P$ and $Q$ are permutation, orthogonal or positive definite matrices then so is also the Kronecker product $P\otimes Q$.

The $k$-th Kronecker power $A^{\otimes k}$ is defined inductively for all positive integers $k$ by 
\beo A^{\otimes 1}=A \p \text{and}\p A^{\otimes k}=A\otimes A^{\otimes (k-1)}\p \text{for }k=2,3,\ldots. \eeo
This definition implies that for $A\in \b M_{m, n}$, we have $A^{\otimes k}\in \b M_{m^k, n^k}$.

For a matrix $X\in \b M_{m,n}$, $vec(X)$ is the $m\cdot n\times 1$ vector defined as
\beo vec(X)=(x_1',\ldots,x_m')'\in \b M_{m\cdot n,1},\eeo  where  $x_i$, $i=1,\ldots,n$ is the $i$-th column of $X$. 

We now derive a generalization of permutation equivalence property, which will be required for the proof of Theorem \ref{mainResult} below.
\begin{lmm}\label{PermutationKroneker}
Let $A_i \in \b M_{p_i, q_i}$ $(i=1,\dots, k)$, $p:=p_1\cdot \ldots \cdot p_k$ and $q:=q_1\cdot \ldots \cdot q_k$. Then, for each permutation $\sigma\in Sym(\left\{1,\ldots,k\right\})$ there exist permutation matrices $P_{\sigma}\in \b M_{p, p}$ and $Q_{\sigma}\in \b M_{q, q}$   such that
\beo A_{\sigma(1)}\otimes \ldots \otimes A_{\sigma(k)}=P_{\sigma} \cdot \left(A_1\otimes \ldots \otimes A_k \right)\cdot Q_{\sigma}. \eeo
\end{lmm}
\proof
Without loss of generality we can assume that $k=4$, for the Kronecker product is associative. Since $(1)\otimes M=M=M\otimes (1)$ for any matrices $M$, it suffices to show that 
$A_1 \otimes A_3 \otimes A_2 \otimes A_4$ is permutation equivalent to  $A_1\otimes A_2 \otimes A_3 \otimes A_4$. 
For a matrix $M$ let $I_{M}$ and $I^M$  denote the identity matrices of such size that one can form the matrix products $I_M \cdot M $ and $M \cdot I^M$. By the property (\ref{PermutationEquivalent}) there exist permutation matrices $P$ and $Q$ with $A_3\otimes A_2=P (A_2\otimes A_3) Q$. Therefore, using (\ref{MixedProductProperty}) we obtain by an easy computation
\beo A_1 \otimes A_3 \otimes A_2 \otimes A_4=\left(I_{A_1} \otimes P \otimes I_{A_4} \right) \cdot \left( A_1 \otimes A_2 \otimes A_3 \otimes A_4 \right) \cdot \left( I^{A_1} \otimes Q \otimes I^{A_4} \right).\eeo
Clearly, both $I_{A_1} \otimes P \otimes I_{A_4}$ and $I^{A_1} \otimes Q \otimes I^{A_4}$ are permutation matrices.  This completes the proof. 
\qed
\\

In the following, let $A=\left[a_{ij}\right]$, $B=\left[b_{ij}\right] \in \b M_{p, q}$ of the same dimensions. The \textit{Hadamard product}, also known as the entrywise product of $A$ and $B$ is denoted by $A\circ B$ and is defined to be the matrix
\beo A\circ B := \left[a_{ij} b_{ij}\right] \in \b M_{p, q}.  \eeo 
The Hadamard product is commutative, associative and distributive w.r.t. addition, and is a principal submatrix of the Kronecker product. 

For a matrix $M$ , let us denote by $\mathbbmtt{1}_M$ the $1$-matrix of the same dimension as $M$, that is, $\mathbbmtt{1}_M=(c_{ij})_{ij}$ with $c_{ij}=1$ for all $i,j$.
We will write it simply $\mathbbmtt{1}$ when no confusion will arise.  It is clear that
\begin{align}
A\otimes B&=(A\otimes \mathbbmtt{1}) \circ (\mathbbmtt{1} \otimes B), \\
B\otimes A&=(\mathbbmtt{1}\otimes A) \circ (B \otimes \mathbbmtt{1}).  
\end{align}
Let $P$ and $Q$ be permutation matrices of such size that one can form the matrix products $P\cdot A$ and $A\cdot Q$. It is easy to check that
\begin{align}
P (A\circ B)  Q&= (PAQ) \circ (PBQ).
\end{align}
\subsection{Permutations on a multiset}\label{PermutationMultiset}
In this section, we generalize the multinomial theorem in terms of Kronecker product instead of the usual multiplication. In order to do this, we first recall the notion of the permutation on a multiset  from \cite[Chapter 1]{stanley0}.

Let $u\in \b N$ and  $k\in \b N_0$. We denote by $C_0(k,u)$ the set of all $u$-compositions of $k$, that is,
\beo C_0(k,u)=\Bigl \{ \lambda\in \b N_0^u: \p \left|\lambda\right|:=\sum_{i=1}^{u}\lambda_i=k \Bigl \}, \eeo
and write $C(k,u)$ instead of $C_0(k,u)\cap \b N^u$. Moreover, we set $M_u:=\left\{1,2,\ldots,u \right\}$. For a $\lambda\in C(k,u)$ a \textit{finite multiset} $Mult(\lambda)$ on the ordered set $M_u$ is a is a set, where $i$ is contained with the multiplicity $\lambda_i$ for all $i\in M_u$.
One regards $\lambda_i$ as the number of repetitions of $i$. A permutation $\pi=(\pi_1\pi_2\ldots\pi_k)$ on $Mult(\lambda)$ can be defined as a linear ordering of the elements of $Mult(\lambda)$, that is, an element $i \in M$ appears exactly $\lambda_i$ times in the permutation $\pi$. 
The set of all permutation on $Mult(\lambda)$ will be denoted by $\f S(\lambda)$. A permutation $\pi= (\pi_1 \pi_2\cdots \pi_k)$ on $Mult(\lambda;M)$ can be regarded as a way to place $k$ distinguishable balls in $u$ distinguishable boxes such that the $i$-th box  contains $\lambda_i$ balls. Indeed, if $i$ $(i=1,\dots, u)$ appears in position $j \in \left\{1,\dots, k\right\}$ of the permutation $\pi$, then we put the ''ball'' $\pi_j$ into the box $i$. For instance let $u=3$, $\lambda:=(1,3,2)\in C(k,u)$ be a $3$-composition of $k=6$ and $\pi=(2 \ 1 \ 2 \ 3 \ 3 \ 2) = :(\pi_1 \  \pi_2 \dots \pi_6)$ be a permutation on $Mult(\lambda)$ then we put $\pi_2$ in the first box, $\pi_1,\pi_3,\pi_6$ in the second box and $\pi_4,\pi_5$ in the third box.  
It is clear that
\beo \left| \mathfrak{S}(\lambda) \right| =\binom{k}{\lambda_1,\dots , \lambda_u}:=\frac{k!}{\lambda_1!\ldots \lambda_u!}.\eeo

Let $m_i\in \b M_{p_i, q_i}$  $(i=1,\ldots, u)$, $\lambda\in C(k,u)$ and $\pi=(\pi_1,\ldots,\pi_k) \in \f S(\lambda)$. We will write $\pi(\row mu)$ instead of $m_{\pi_1} \otimes m_{\pi_2} \otimes \dots  \otimes m_{\pi_k}$.
Moreover, we set 
\beo W(n,u):=\left\{\mu=(\mu_1,\dots,\mu_u)\in \left\{1,\ldots,n \right\}^u: \mu_1<\mu_2<\dots <\mu_u \right\}. \eeo
In the following theorem, which will be used in Section \ref{eukRad_P4} several times, we expand a Kronecker power of a matrix sum in terms of powers of the terms in that sum.
\begin{trm}\label{TrmMNTkreuz}
Let $k\in \b N$ and $x_1,\dots, x_n \in \b M_{p,q}$. Then
\be\label{MNTkreuz} \Bigl(\sum\limits_{i=1}^{n} x_i \Bigl)^{\otimes,k}=\sum\limits_{u=1}^{k} \sum\limits_{\lambda\in C(k,u)}\sum\limits_{\mu\in W(n,u)} \sum\limits_{\pi\in \mathfrak{S}(\lambda )} \pi(x_{\mu_1},\dots,x_{\mu_u}). \ee
\end{trm}
For $p=q=1$ the Kronecker product coincides with the usual multiplication on $\b R$ and therefore, (\ref{MNTkreuz}) generalizes multinomial formula.
For indices $u\in \left\{1,\ldots,k\right\}$, $\mu=(\mu_1,\ldots,\mu_n)\in W(n,u)$,  $\lambda\in C(k,u)$ and $\pi\in  \f S(\lambda)$ let us consider the associated summand 
\be\label{SummandPiX} \pi(x_{\mu_1},\dots,x_{\mu_u})=x_{\mu_{\pi_1}}\otimes \ldots \otimes x_{\mu_{\pi_k}} \ee
from (\ref{MNTkreuz}). It is clear that the different matrices $x_{\mu_1},\ldots,x_{\mu_u}$, the numbers of their repetitions and their exact positions in the Kronecker product (\ref{SummandPiX})  are described by $\mu=(\mu_1,\ldots,\mu_u)\in W(n,u)$, $\lambda=(\lambda_1,\ldots, \lambda_u)\in C(k,u)$ and $\pi=(\pi_1,\ldots, \pi_k)\in \f S(\lambda)$  respectively. 
\proof
We proceed by induction on $k$. For $k=1$ there is nothing to proof. Next suppose as induction hypothesis that (\ref{MNTkreuz}) holds with $k-1$ instead of $k$. It gives 
\begin{align}\notag  \Bigl(\sum\limits_{i=1}^{n} x_i \Bigl)^{\otimes,k}
&= \sum\limits_{u=1}^{k-1} \sum\limits_{\lambda\in C(k-1,u)}\sum\limits_{\mu\in W(n,u)} \sum\limits_{\pi\in \mathfrak{S}(\lambda)} \pi(x_{\mu_1},\dots,x_{\mu_u}) \otimes \sum\limits_{j=1}^{n} x_j
\\ \label{MNTkreuzIV} &= \sum\limits_{j=1}^{n} \sum\limits_{u=1}^{k-1} \sum\limits_{\lambda\in C(k-1,u)}\sum\limits_{\mu\in W(n,u)} \sum\limits_{\pi\in \mathfrak{S}(\lambda)} \pi(x_{\mu_1},\dots,x_{\mu_u})\otimes x_j. \end{align}
Consider a term $\pi(x_{\mu_1},\dots,x_{\mu_u}) \otimes x_j$ of the sum above, that is, $j\in\left\{1,\dots, n\right\}$, $u\in \left\{1,\dots, k-1\right\}$, $\lambda\in C(k-1,u)$, $\mu\in W(n,u)$ and $\pi\in \mathfrak{S}(\lambda)$.  
If there is $\beta\in \left\{1,\dots, u\right\}$ with $j=\mu_{\beta}$ then it corresponds to exact one summand in (\ref{MNTkreuz}) associated with indices $\tilde{u}=u$, $\tilde{\lambda}=(\lambda_1,\dots,\lambda_{\beta-1},\lambda_{\beta}+1, \lambda_{\beta+1},\dots,\lambda_u)$, $\tilde{\mu}=\mu$ and $\tilde{\pi}=(\pi_1,\ldots,\pi_{k-1}, \beta)$. 
In the other case, that is, if $j\in (\mu_{\beta-1},\mu_{\beta})$ for an $\beta\in \left\{1,\ldots,u+1\right\}$ with the convention $\mu_0:=0$ and $\mu_{u+1}=\infty$ the term  $\pi(x_{\mu_1},\dots,x_{\mu_u}) \otimes x_j$ corresponds to a summand in (\ref{MNTkreuz}) associated with indices $\tilde{u}=u+1$, $\tilde{\lambda}=(\lambda_1,\dots,\lambda_{\beta-1}, 1, \lambda_{\beta}, \dots,\lambda_u)$, $\tilde{\mu}=(\mu_1,\ldots,\mu_{\beta-1},j,\mu_{\beta},\ldots,\mu_{u})$ and $\tilde{\pi}=(\pi_1,\ldots,\pi_{k-1}, \beta)$. 
As the number of summands in both (\ref{MNTkreuz}) and (\ref{MNTkreuzIV}) is equal to $n^k$, the induction step follows.
\qed
\\

In the following we collect some known facts about multivariate Bessel functions on the cone $\Pi_q$, which will be needed later. The material is mainly taken from \cite{roesler}. We also refer to $\cite{FK}$ and $\cite{herz}$.

\subsection{Bessel functions on the cone $\Pi_q$}\label{BesselFunctions}
Let $Z_\lambda$ denote the \textit{zonal polynomials}, which are indexed by partitions $\lambda=(\lambda_1\geq \lambda_2\geq \ldots \geq \lambda_q)\in \b N_0^q$ (we write $\lambda\geq 0$ for short) and normalized such that
\beo tr(x)^k=\sum\limits_{|\lambda|=k} Z_{\lambda}(x) \p \forall \ k\in \b N_0; \eeo
see \cite{FK} for the construction of $Z_{\lambda}$ and further details. It is well known that the $Z_{\lambda}$ are homogeneous polynomials which are invariant under conjugation by $\b O_q$ and thus depend only on the eigenvalues of their argument. More precisely, for $x\in \b H_q$ with eigenvalues $\xi=(\xi_1,\ldots,\xi_q)\in \b R^q$, one has 
\beo Z_{\lambda}(x)=C_{\lambda}^{\alpha}(\xi) \p \text{with}\p \alpha=2\eeo 
where the $C_{\lambda}^{\alpha}$ are the \textit{Jack polynomials} of index $\alpha$ in a suitable normalization (see \cite{FK},\cite{roesler}). The Jack polynomials $C_{\lambda}^{\alpha}$ are homogeneous of degree $\left|\lambda\right|$ and symmetric in their arguments. 
Let $\alpha>0$ be a fixed parameter. For partitions $\lambda=(\lambda_1,\ldots,\lambda_q)$ we introduce the \textit{generalized Pochhammer symbol}
\beo (\mu)_{\lambda}^{\alpha}=\prod\limits_{j=1}^{q}\left(\mu-\frac{1}{\alpha}(j-1)\right)_{\lambda_j}\p (\mu\in \b C),\eeo where $(\cdot)_j$ denotes the usual Pochhammer symbol. 
For an index $\mu\in \b C$ satisfying $(\mu)_{\lambda}^{\alpha}\neq 0$ for all $\lambda\geq 0$ the \textit{matrix Bessel functions} associated with the cone $\Pi_q$ are defined as ${}_0 F_1$-hypergeometric series in terms of the $Z_{\lambda}$, namely
\be\label{BesselfunctionMu} J_{\mu}(x)=\sum_{\lambda \geq 0} \frac{(-1)^{|\lambda|}}{(\mu)_{\lambda}^{d/2} |\lambda|!}Z_{\lambda}(x). \ee
For a general background on matrix Bessel functions, the reader is referred to the fundamental article \cite{herz}.
If $q=1$, then $\Pi_q=[0,\infty[$ and we have $\c J_{\mu}(x^2/4)=j_{\mu-1}(x)$, where $j_{\kappa}(z)= {}_0 F_1(\kappa+1; -z^2/4)$ is the modified Bessel function in one variable.

\subsection{Polynomials on $\b M_{p, q}$}\label{PolynomialsOnMpq}
Let $p$, $q\in \b N$. For $\kappa =(\kappa_{ij})_{i,j} \in \b N_0^{p\times q}$ (a composition) we set $\left|\kappa\right|:=\sum_{i,j}\kappa_{ij}$ and $R_i(\kappa):=\sum_{j=1}^{q}\kappa_{ij}$, $i=1,\ldots,p$. Moreover, we write $z^{\kappa}:=\prod_{i,j}z_{ij}^{\kappa_{ij}}$. Clearly, $z^{\kappa}$ is a monomial of degree $\left| \kappa \right|$. The spaces of \textit{polynomials} and \textit{row-even polynomials} are defined by
\begin{align*}
\c P&:=span\Big \{x^\kappa: \  \kappa\in \b N_0^{p\times q} \Big \},\\
\c P_{e}&:=span\Big\{x^\kappa: \  \kappa\in \b N_0^{p\times q}, \forall \ i \ R(i)  \text{ is even } \Big\}
\end{align*} respectively. 

We shall need the following observation:
\begin{lmm}\label{even_polynomials}
Let $r\in \Pi_q$, and $\kappa\in \b N_0^{p\times q}$. Then 
\beo \Psi_{r,\kappa}:\b M_{p,q}\to \b R, \p \Psi_{r,\kappa}(z):=((zr)'(zr))^{\kappa} \eeo
is an even polynomial of degree $2|\kappa|$.
\end{lmm}
\proof
Since the product of two row-even polynomials is also a row-even polynomial, the proof follows easily by induction on $n=|\kappa|$. 
\qed

\subsection{Radial measures on $\b M_{p,q}$ and their moments}\label{RadialMeasuresOnMpq}
In this section we study radial measures on the space $\b M_{p,q}$. In particular, we derive asymptotic results for their moments as $p\to \infty$. This results will play a key role in the proof of Theorem \ref{mainResult}. We start with the definition of a radial measure on $\b M_{p,q}$.
\begin{dfn}
A measure $\nu_{p}$ on $\b M_{p,q}$ is called \textit{radial} if 
\beo A(\nu_{p})=\nu_{p} \quad \forall \ A \in \b O_p, \eeo that is, if it is invariant under the action (\ref{OperationOnMpq}). In particular, for $q=1$ a measure $\nu_p$ on $\b R^p$ is radial if it is invariant under rotations.
\end{dfn} 
\begin{rem} It is well known that for each probability measure $\nu \in \c M ^1(\Pi_q)$ and a dimension $p\in \b N$ there is a unique radial probability measure $\nu_{p}\in \c M^1(\b M_{p,q})$ with $\nu$ as its radial part, that is, $\varphi_p(\nu_p)=\nu$. \end{rem}

In order to study radial measures on $\b M_{p,q}$ and their moments we need an analogue of a sphere in our higher rank setting. For an $r\in \Pi_q$ we define a sphere of radius $r$ as the set
\beo \Sigma_{p,q}^{r}=\left\{x \in \b M_{p,q}: \ \sqrt{x'x} = r \right\}. \eeo 
Clearly, $\Sigma_{p,q}^{r}$ is the orbit of the block matrix $\sigma_r:=(r \  0)' \in \b M_{p,q}$ according to the operation (\ref{OperationOnMpq}). For simplicity of notation, we write $\Sigma_{p,q}$ instead of $\Sigma_{p,q}^{I_q}$, where $I_q\in \b R^{q\times q}$ denotes the identity matrix. In the case $q=1$ we identify $\Sigma_{p,1}^{r}$ with the Euclidean sphere of radius $r\in [0,\infty[$. Moreover, let us denote by $U_p^{r}$ the uniform distribution on a sphere $\Sigma_{p,q}^{r}$. 

One can easily show that a radial probability measure $\nu_{p}$ with its radial part $\nu \in \c M ^1 (\Pi_q)$ enables the decomposition
\be\label{Decomposition-nup-nu} \nu_{p}(\cdot) = \int_{\b M_{p,q}} U_p^{\varphi_p(x)}(\cdot ) d\nu_p (x)=\int_{\Pi_q} U_p^{r}(\cdot ) d\nu (r)\in \c M^1(\b M_{p,q}). \ee
In the sense of Jewett \cite{jewett}, the formula above is an example of a decomposition  of  a measure (here $\nu_p$) according to so called orbital morphism (here $\varphi_p$). More precisely, $\varphi_p$ is an orbital mapping, that is a proper and open continuous surjection from $\b M_{p,q}$ onto $\Pi_q$. The mapping  $r\mapsto U_p^{r}$ from $\Pi_q$ to $\c M^1(\b M_{p,q})$ is a recomposition of $\varphi_p$ which means that each $U_p^r$ is a probability measure on $\b M_{p,q}$ with support equal to $\varphi_p^{-1}(r)$ (here $=\Sigma_{p,q}^{r}$), and such that $\nu_{p}= \int_{\b M_{p,q}} U_p^{\varphi_p(x)} d\nu_p (x)$. 

\begin{dfn}
Let $Z$ be a $\b M_{p,q}$-valued random variable with distribution $\mu\in \c M^1(\b M_{p,q})$. We say that $\mu\in \c M^1(\b M_{p,q})$ (or $Z$) admits a $k$-th moment $(k\in \b N_0)$ if $\int_{\b M_{p,q}} \left\|z\right\|^k d\mu(z)<\infty$,
and define in this case the \textit{$k$-th moment} of $\mu$ (or $Z$) by 
\beo M_k(\mu)\Doteq M_k(Z)\Doteq \b E \left(Z^{\otimes,k}\right)\in \b M_{p^k,q^k}. \eeo 
Let $I=\left\{(i_1,j_1),\dots, (i_k,j_k)\right\}$ with $i_{\alpha}\in \left\{1,\ldots,p \right\}$ and $j_{\alpha} \in \left\{1,\ldots,q\right\}$ for $\alpha\in \left\{1,\ldots, k\right\}$. 
Then the $I$-th component $M_k(Z)_I$ of $M_k(Z)$ is given by 
\beo M_k(Z)_{I}=\b E \bigl( Z_{i_{1},j_{1}}\cdot\ldots \cdot Z_{i_{k},j_{k}} \bigl). \eeo
Moreover, for an $\kappa\in \b N_0^{p\times q}$ with $\left|\kappa\right|=k$ we set 
\beo m_\kappa(\mu):=\int_{\b M_{p,q}} z^\kappa d\mu(z)\in \b R, \eeo and call $m_{\kappa}(\mu)$ also the \textit{$\kappa$-th moment} of $\mu$. 
\end{dfn}
In the following $\widehat{\mu}$ denote the characteristic function of a probability measure $\mu$ on $\b M_{p,q}$, that is, 
\beo \widehat{\mu}(x)=\int_{\b M_{p,q}}\exp(i\left\langle x,y \right\rangle)d\mu(y). \eeo 
Let $k\in\b N_0$ and $\kappa\in \b N_0^{p\times q}$ with $|\kappa|=k$. If $\mu$ admits a $k$-th moment then we have
\be\label{relationCharMom} m_{\kappa}(\mu)=(-i)^{\left|\kappa\right|} D_{\kappa} \widehat{\mu} (x)|_{x=0}, \ee
where  $D_{\kappa}$ is the differential operator $\frac{\partial^{\kappa_{ij}}}{\partial x_{11}^{\kappa_{11}}} \frac{\partial^{\kappa_{12}}}{\partial x_{12}^{\kappa_{12}}} \cdots \frac{\partial^{\kappa_{pq}}}{\partial x_{pq}^{\kappa_{pq}}}$.

Here and subsequently, $\nu_p$ denotes a radial probability measure on $\b M_{p,q}$ with the corresponding radial part $\nu \in \c M^1(\Pi_q)$ and $X$ is a $\b M_{p,q}$-valued random variable with radial distribution $\nu_p$.

In the next lemmas we explore the covariance structure of $X$ and compute the asymptotic behaviour of the moments of $\nu_p$ for large dimensions $p$.
\begin{lmm}\label{CovStructure_nup}
Let $X=(X_{ij})_{i,j}$ be $\b M_{p,q}$-valued random variable  with radial distribution $\nu_p \in \c M^1(\b M_{p,q})$. Then
\be\label{EW_VarXij} \b E \left(X \right)=\mathbf{ 0 } \p \text{and} \p \b E\left(X_{ji}X_{lk}\right)=\delta_{j,l} \b E \left(X_{1i}X_{1k}\right)\ee
\end{lmm}
\proof
For $r\in \b R\setminus\{0\}$ let $M_{j,r}$ and $S_{i,j}$ be $p\times p$ matrices produced by multiplying all elements of row $j$ of the identity matrix by $r$ and by exchanging row $i$ and row $j$ of the identity matrix respectively. As $S_{i, j}$ is a symmetric involution on $\b M_p$, we have $S_{i, j}\in \b O_p$. For $r=\pm 1$ the matrix $M_{j,r}$ is also orthogonal. 
By assumption, $X$ and $AX$ are identically distributed for any $A \in \b O_p$. Therefore, we have 
\beo  \b E\left(X\right)_{ij}=\b E \left( M_{j,-1}X \right)_{ij}= - \b E \left( X \right)_{ij}. \eeo So the first equality in (\ref{EW_VarXij}) holds.

Choose $i,k\in \{1,\dots , q\}$ and $j,l\in\{1,\dots , p\}$ with $j\neq l$.  We conclude from 
\beo \b E\left(X_{ji}X_{lk}\right)=\b E \left( (M_{j,-1}X)_{ji}(M_{j,-1}X)_{lk} \right)=\b - E \left( X_{ji}X_{lk} \right)\eeo
that $\b E(X_{ji}X_{lk})=0$. We now turn to the case $j=l$.  The transformation $\b M_{p,q}\to \b M_{p,q}$, $A\mapsto S_{i, j}A$, switches all matrix elements on row $i$ with their counterparts on row $j$. Therefore, from radiality of $P_X=\nu_p$ it follows that 
$$ \b E\left(X_{ji}X_{jk}\right)=\b E \left( \left(S_{j, 1}X\right)_{ji}\left(S_{j, 1}X\right)_{jk} \right) = \b E \left( X_{1i}X_{1k} \right)\quad \text{for }i,k\in \{1,\dots , q\}. $$
\qed 
\\
Now let us denote by $x_i$ the $i$-th row of $X$. According to the lemma above, we have
\beo  \b Cov(x_i,x_j)=\delta_{i,j} \cdot \b E(x_1 x_1')=:T_p \in \b M_q. \eeo
Therefore, we obtain
\beo \b Cov(X):=\b Cov(vec(X'))=I_p\otimes T_p\in \b M_{q\cdot p}. \eeo

\begin{lmm}
The characteristic function for the uniform distribution $U_p^{r}$ on the sphere $\Sigma_{p,q}^{r}$ of radius $r \in\Pi_q$ is given by
\be\label{GV_char} \widehat{U_p^{r}} (z)=J_{\mu}\left(\frac{1}{4}(zr)'(zr)\right), \p (z\in \b M_{p,q})\ee 
where $\mu=\frac{p}{2}$ and $J_{\mu}$ is the Bessel function of index $\mu$ of Eq. (\ref{BesselfunctionMu}).
\end{lmm}
\proof
Let $r\in \Pi_q$. Consider the map 
\beo T_r:\Sigma_{p,q} \to \Sigma_{p,q}^{r},\quad y\mapsto yr.\eeo
Since $T_r(U_p^{I_q})=U_p^{r}$,  we get by substitution formula
\beo  \widehat{U_p^{r}} (z)=\int_{\b M_{p,q}}e^{i \langle z,y\rangle} dU_p^{r}(y)=\int_{\Sigma_{p,q}}e^{i \langle z,yr\rangle} dU_p^{I_q}(y). \eeo
On the other side, according to Proposition XVI.2.3. of \cite{FK} we have for $x\in \b M_{p,q}$ the identity
\beo \int_{\Sigma_{p,q}}e^{i \langle y,x\rangle}dU_p^{I_q}(y)=J_{\mu}\left(\frac{1}{4}x'x\right),\quad \mu=\frac{p}{2}. \eeo
By taking these two identities above into account, (\ref{GV_char}) follows as claimed.
\qed
\begin{lmm}\label{MomentsOfU}
Let $\kappa \in \b N_0^{p\times q}$, $l:=|\kappa|/2$ and $\mu=\frac{p}{2}$. The $\kappa$-th moment $m_{\kappa}(U_p^r)$ of the uniform distribution on $\Sigma_{p,q}^{r}$ is given as follows:
\begin{itemize}
\item[(a)] If $R_i(\kappa)=\sum_{j=1}^{q}\kappa_{ij}$ is even for all $i=1,\ldots, p$, then $l\in \b N_0$ and
\be\label{k-thMomentOf-U-pr} m_{\kappa}(U_p^r)=\frac{1}{4^{l}\left| \kappa \right|!} \sum\limits_{ \lambda\in C_0(l,q)} \frac{1}{(\mu)_{\lambda}^{d/2}} D_{\kappa} \Bigl(Z_{\lambda}\big((zr)^{*}(zr)\big) \Bigl)_{\big | z=0}.
\ee
\item[(b)] If $R_i(\kappa)$ is not even for some $i=1,\ldots,p$, then $m_{\kappa}(U_p^r)=0$.
\end{itemize}
\end{lmm}
\proof
By the Identity (\ref{relationCharMom}), the preceding lemma and (\ref{BesselfunctionMu}) we have 
\be\label{kthMomentOfUprH} m_{\kappa}(U_p^r)=(-i)^{\left| \kappa \right|} \sum\limits_{j=0}^{\infty}\frac{(-1)^j}{j!}\sum\limits_{\lambda\in C_0(j,q)}\frac{1}{(\mu)_{\lambda}^{d/2}} D_{\kappa} \left(Z_{\lambda} \left(\frac{1}{4} (zr)^{*}(zr)\right)\right)_{\big | z=0}.\ee
Let $\lambda \in \b N_0^q$ and $p_r:z\mapsto Z_{\lambda}\left((zr)^{*}(zr)\right)$. Since $Z_{\lambda}$  is a homogeneous polynomial of degree $|\lambda|$, Lemma \ref{even_polynomials} shows that $p_r$ 
is a homogeneous, row-even polynomial of degree $2\left|\lambda\right|$.  Therefore, each term on the right-hand side of (\ref{kthMomentOfUprH}) 
vanishes if $\kappa \in \b N_0^{p\times q}$ with $R_i(\kappa)$ is odd for some $i\in \left\{1,\ldots,p\right\}$ or if $\left| \kappa \right| \neq 2 \left|\lambda \right|$. This proves the assertion.
\qed
\begin{trm}\label{kMomentof_nu_p}
Let $\kappa\in \b N_0^{p\times q}$, $l:=|\kappa|/2$, $\nu\in \c M^1(\Pi_q)$ and $\nu_p \in \c M^1(\b M_{p,q})$ be the corresponding radial probability measure on $\b M_{p,q}$ which admits a $\kappa$-th order moment. Then the $\kappa$-th moment $m_{\kappa}(\nu_p)$ of $\nu_p$ exists in $\b R$ and has the following asymptotic as $p\to \infty$:
\begin{itemize}
\item[(a)] If $R_i(\kappa)$ is even for all $i=1,\ldots, p$, then $m_{\kappa}(\nu_p)=O\left(\frac{1}{p^l}\right)$.
\item[(b)] If $R_i(\kappa)$ is not even for some $i=1,\ldots,p$, then $m_{\kappa}(\nu_p)=0$.
\end{itemize}
\end{trm}
\proof
The existence of $m_{\kappa}(\nu_p)$ is clear. By the decomposition (\ref{Decomposition-nup-nu}) we obtain 
\beo m_{\kappa}(\nu_p)=\int_{\Pi_q}m_{\kappa} (U_p^r) d\nu(r), \eeo  
where $U_p^r$ is the uniform distribution on $\Sigma_{p,q}^{r}$. 
Therefore, the assertion (b) follows immediately from Lemma \ref{MomentsOfU} (a). Now we turn to the case (a). Since the $\lambda$-th term in the sum (\ref{k-thMomentOf-U-pr}) is a homogeneous polynomial in the variable $r_{11},r_{12},\dots, r_{qq}$ of degree $2\left|\lambda\right|$ which is also independent of $p$, Lemma \ref{MomentsOfU} (a) leads to 
\beo m_{\kappa}(\nu_p)=\sum\limits_{\lambda\in C_0(l,q)} \int_{\Pi_q} O\left(\frac{1}{p^l}\right)d\nu(r)=O\left(\frac{1}{p^l}\right). \eeo
\qed

\subsection{Matrix variate normal distribution and their moments}\label{MVNDistribution}
In this section we derive some results concerning the class of matrix variate normal distribution on $\b M_q$, to which belongs the limiting distribution in our main result \ref{mainResult1}. 

Let $Z=(z_{ij})_{1\leq i,j\leq q}$ be a real matrix variate normal distributed variable with mean matrix $\mu\in \b M_{q}$ and symmetric covariance matrix
\be\label{symCovMatrice} \Sigma= \bigl(\Sigma_{(i,j),(l,k)} \bigl)_{1\leq i,j,l,k\leq q}=\bigl(\Sigma_{(l,k),(i,j)} \bigl)_{1\leq i,j,l,k\leq q}\in \b M_{q^2}\cong \b M_q\otimes \b M_q.\ee
We write $Z\sim \c N(\mu,\Sigma)$ for short. This means that $vec(Z')$ is $\c N(vec(\mu'),\Sigma)$-distributed. In order to prove some formulas for moments $M_k(Z)=\b E(Z^{\otimes k})$ of $Z$, which we will use in Section \ref{eukRad_P4}, we need the following notation. 
Let $u\in \b N$, $k:=2u$, $I=((i_1,j_1),\ldots , (i_k,j_k)) \in (\left\{1,\ldots,q\right\}^2)^{k}$, $\lambda=(2,\ldots,2)\in C(k,u)$ and $\pi=(\pi_1,\ldots,\pi_k)\in \f S(\lambda)$. For a tuple $v=(v_1,\ldots, v_n)$ we will write $\left\{v\right\}$ instead of the set $\left\{v_1,\ldots, v_n\right\}$. 
Consider the sets 
\beo \pi(I)_i=\left\{(i_\mu,j_\mu)\in \left\{I \right\}:\ \pi_\mu=i \right\}\p (i=1,\ldots,u). \eeo Obviously $\pi(I)_i$ $(i=1,\ldots, u)$ forms a partition of $\left\{I\right\}$ with $\left|\pi(I)_i\right|=2$. We define for $\pi$, $I$ and a symmetric covariance matrix $\Sigma$ as in (\ref{symCovMatrice}),
\beo \pi(\Sigma)_{I}:=\prod_{i=1}^{u} \Sigma_{(\alpha_i,\beta_i),(\gamma_i,\delta_i)} \text{ where } \left\{(\alpha_i,\beta_i),(\gamma_i,\delta_i)\right\}=\pi(I)_i.\eeo
For instance let $u=2$, $I=\left\{(2,1), (2,2), (3,2), (2,1) \right\}$, $\lambda:=(2,2)\in C(4,2)$ and $\pi=(1 \ 2 \ 1 \ 2) = :(\pi_1 \dots \pi_4)$; then we have $\pi(I)_1=\left\{ (2,1), (3,2) \right\}$, $\pi(I)_2=\left\{ (2,2), (2,1) \right\}$  and $\pi(\Sigma)_I=\Sigma_{(2,1),(3,2)}\cdot \Sigma_{(2,2),(2,1)}$. \\

The moment formulas $M_k(Z)$ for multivariate normal distributed random vector $Z \sim \c N(\mu,\Sigma)$ are well studied in the literature (see \cite{Triantafyllopoulos} and \cite{GuptaNagar}). In \cite[Theorem 1]{Triantafyllopoulos} we find moment formulas for centered Gaussian distribution $Z$, which are derived in a relative fast and elegant way. This formula can be easily translated in our setting. Namely, the $I$-th component of $k$-th order moment of a $\c N(\mathbf{0},\Sigma)$-distributed random matrix $Z$ is given by
\be\label{MomentFormulaMND} M_k(Z)_I= \begin{cases}  0,  & \text{if $k$ is odd} , \\ \frac{1}{u!}\sum\limits_{\pi \in \f S(\lambda)} \pi(\Sigma)_I, & \text{if $k=2u$, $\lambda=(2,\ldots,2)\in C(k,u)$}. \end{cases}\ee
In the most classical case $q=1$ , that is, $Z$ is centered Gaussian distribution on $\b R$ with variance $\sigma^2>0$ the identity (\ref{MomentFormulaMND}) reduces to the well known formula 
\be \b E(Z^k)= \begin{cases}  0,  & \text{if $k$ is odd} , \\ \sigma^k (k-1)(k-3)\cdot \ldots \cdot 3 \cdot 1 &\text{if $k$ is even}. \end{cases}\ee

The following two simple observations concerning the $k$-th moment of normal distributed random matrix and a sum of two independent, normal distributed random matrices respectively will be needed for the proof of Theorem \ref{mainResult}.
\begin{lmm}\label{MomenteNormalDistribution}
Let $Z$ be $\c N(\mathbf{0},\Sigma)$-distributed random variable and $Z_1,Z_2,\ldots$ independent copies of $Z$. The $k$-th order moment of $Z$ is given by
\beo 
M_k(Z)= \begin{cases}  \mathbf{0},  & \text{if $k$ is odd} , \\ \frac{1}{u!}\sum\limits_{\pi \in \f S(\lambda)} \b E \pi(Z_1,\ldots,Z_u), 
& \text{if $k=2u$} \end{cases}
\eeo
where  $\lambda=(2,\dots, 2)\in C(2u,u)$.
\end{lmm}
\proof 
Let $k\in \b N$ and $I=\left((i_1,j_1),\ldots, (i_k,j_k) \right) \in (\left\{1,\ldots,q\right\}^2)^{k}$. 
If $k$ is odd, then it follows by (\ref{MomentFormulaMND}) that $M_k(Z)_I=0$. Suppose that $k=2u$, $(u\in \b N)$.  
For $\pi\in \f S(\lambda)$, $\lambda=(2,\ldots,2)\in C(k,u)$ and $I$ as above, we have $\pi(\row Zu)_I=\left(Z_{\pi_1}\otimes \ldots \otimes Z_{\pi_k}\right) _I$. 
Let $\left\{(\alpha_i,\beta_i),(\gamma_i,\delta_i) \right\}=\pi(I)_i$, $i=1,\ldots,u$. 
By independence it follows
\begin{align*} \b E \pi(\row Zu)_I & = \b E (Z_{\pi_1}\otimes \ldots \otimes Z_{\pi_k})_I = \prod\limits_{i=1}^{u}\b E\bigl(Z_i\otimes Z_i \bigl)_{(\alpha_i,\beta_i),(\gamma_i,\delta_i) } \\ & = \prod\limits_{i=1}^{u} \Sigma_{(\alpha_i,\beta_i),(\gamma_i,\delta_i)} = \pi(\Sigma)_I.\end{align*} 
The lemma is now a consequence of Eq. (\ref{MomentFormulaMND}).
\qed
\begin{lmm}\label{Z1Z2ZNormaldistributions}
Let $Z_i$ $(i=1,\ 2)$ be independent random variables with distributions $\c N(\mathbf{0},\Sigma_i)$. Then 
\be\label{Z1Z2Z} \b E\left(\bigl(Z_1+Z_2\bigr)^{\otimes,k}\right)= \sum\limits_{l=0}^{k}\sum\limits_{\pi\in \f S \left((l,k-l)\right)} 
\b E \pi(Z_1,\mathbbmtt{1}) \circ \b E \pi(\mathbbmtt{1},Z_2). 
\ee
\end{lmm}
\proof  By the definition of $\circ$-product and independence of $Z_1$ and $Z_2$ we have
\begin{align*} \b E \left((Z_1+Z_2)^{\otimes,k}\right)&=\sum\limits_{l=0}^{k} \sum\limits_{\pi\in \mathfrak{S}((l,k-l) )} \b E \pi(Z_{1},Z_{2} )  \\
&=\sum\limits_{l=0}^{k} \sum\limits_{\pi\in \mathfrak{S}((l,k-l) )} \b E \left(\pi(Z_{1}, \mathbbmtt{1}) \circ \pi(\mathbbmtt{1},Z_2) \right) \\
&=\sum\limits_{l=0}^{k} \sum\limits_{\pi\in \mathfrak{S}((l,k-l) )} \b E \pi(Z_{1},\mathbbmtt{1}) \circ \b E \pi(\mathbbmtt{1},Z_2). 
\end{align*}
\qed

\section{Radial limit theorems on  $\b M_{p,q}$ for $p\to \infty$ }\label{eukRad_P4}
Let $\nu\in \c M^1(\Pi_q)$ be a fixed probability measure such that $\int_{\Pi_q}\left\|x\right\|^4 d\nu(x)<\infty$.  Then for each dimension $p\in \b N$ there is a unique radial probability measure $\nu_{p}\in \c M^1(\b M_{p,q})$ with $\nu$ as its radial part, that is, $\nu=\varphi_{p}(\nu_{p})$. Let $X=(x_{ij})_{ij}$ be $\nu_p$ distributed random matrix on $\b M_{p,q}$. 
We define
\begin{align*}
r_{2}(\nu)&:=\b E \left(\varphi_p^2(X) \right)=p\cdot T_p \in \Pi_q, \\ 
\Sigma(\nu)&:= \b Cov(\varphi_p^2(X)) =\b Cov(vec(\varphi_p^2(X)'))
\in \Pi_{q^2} \cong \Pi_q\otimes \Pi_q.
\end{align*}
Clearly, $r_2(\nu)$ and $\Sigma(\nu)$ are independent from $p$. Now, we consider for each $p\in \b N$ i.i.d. $\b M_{p,q}$-valued random variables 
\beo X_k:=\left( X_k^{(i,j)} \right)_{ 1\leq i\leq p, \  1\leq j \leq q}, \p k\in \b N 
\eeo 
with law $\nu_{p}$ as well as the random variables 
\be\label{mainProcess} \Xi_n^p(\nu):=\varphi_{p} ( S_n^p )^2-nr_2(\nu), \ee where $S_n^p:=\sum_{k=1}^{n} X_{k}$.
Let $(p_n)_{n\in \b N}\subset \b N$ be a sequence with $\lim_{n\to \infty}p_n=\infty$. In this section, we derive the following two complementary CLTs for $\b M_{q}$-valued random variables 
$\Xi_n(\nu):=\Xi_n^{p_n}(\nu)$ under disjoint growth conditions for the dimensions $p_n$.
\begin{trm}\label{mainResult}
Assume that $\nu \in \c M^1(\Pi_q)$ admits finite fourth moment. \\
\textbf{CLT I:} If $\lim\limits_{n\to\infty}\frac{n}{p_n}=\infty$, then $\frac{\sqrt{p_n}}{n}\cdot \Xi_n(\nu)$ tends in distribution to the centered matrix variate normal distribution $\c N(\mathbf{0},T(\nu))$ with covariance matrix $T(\nu):=T_1(\nu)+T_2(\nu)$ where 
\be\label{T1T2} T_1(\nu)_{(i,j),(k,l)}=r_2(\nu)_{i,k}r_2(\nu)_{j,l}\quad \text{and}\quad T_2(\nu)_{(i,j),(k,l)}=r_2(\nu)_{i,l}r_2(\nu)_{j,k}. \ee
\\
\textbf{CLT II:} If $\lim\limits_{n\to\infty}\frac{n}{p_n}=c\in [0,\infty[$, then $\frac{1}{\sqrt{n}}\cdot\Xi_n(\nu)$ tends in distribution to the centered matrix variate normal distribution $\c N(\mathbf{0},\Sigma(\nu)+cT(\nu))$ (where $T(\nu)$ is given as in CLT I.) 
\end{trm}
Notice that for $q=1$ we obviously have $\nu\in \c M^1([0,\infty[)$, $r_2(\nu)=\int_0^{\infty}x^2d\nu(x)$, $T(\nu)=2r_2(\nu)^2$ and $\Sigma(\nu)=\int_0^{\infty}x^4 d\nu(x)-r_2(\nu)^2$. Therefore, Theorem \ref{mainResult} completely agrees with Theorem \ref{mainResult0}.

The proof of Theorem \ref{mainResult} will be divided into two main steps: In the first step we prove a reduced form of Theorem \ref{mainResult} assuming that $\nu$ has a compact support. In the second step we will show how to get rid of the support condition for $\nu$. Both steps are based on the decomposition of $\Xi_n(\nu)$ via
\begin{align}
\label{DefA_n} 							 \f A_{n}(\nu)&:=\sum\limits_{i=1}^{n} A_i, \quad \text{with } A_i:=\varphi_{p_n}(X_{i})^2-r_2(\nu), \\
\label{DefB_n} \text{and} \p \f B_{n}(\nu)&:=\sum\limits_{i=1}^{p_n}B_{i},  \quad \text{with }B_i:=\sum\limits_{\alpha,\beta=1,\ldots,n; \ \alpha \neq \beta} \left[X_{\alpha}^{(i,j)} X_{\beta}^{(i,l)}\right]_{1\leq j,l\leq q}.
\end{align}
We compute the covariance structure of $\f A_n(\nu)$ and $\f B_n(\nu)$ respectively:
Since the random variables $A_i$ $(i=1,2,\ldots)$ are independent and identically distributed, it is easily seen that
\be\label{A_structure} \b E( A_k)=\mathbf{0}, \quad \b Cov(A_i,A_j)=\delta_{i,j}\Sigma(\nu). \ee
This gives
\be \frac{1}{n} \b Cov(\f A_n(\nu))= \frac{1}{n}\sum\limits_{k=1}^{n} \b Cov( A_k ) = \Sigma(\nu). \ee
By the independence of random variables $X_{k}$, $k\in \b N$ and Lemma \ref{CovStructure_nup} we obtain
\be \b E( B_k)=\mathbf{0}, \quad \b Cov(B_i,B_j)=\delta_{i,j} \frac{n(n-1)}{p_n^2}T(\nu). \ee
We thus get
\be \lim\limits_{n\to \infty} \frac{p_n}{n^2} \b Cov(\f B_n(\nu))=\lim\limits_{n\to \infty} \frac{p_n}{n^2}\sum\limits_{i=1}^{p_n} \b Cov( B_i ) = T(\nu). \ee
In the following we will establish convergence in distribution of the random variables $\f A_n(\nu)$ and $\f B_n(\nu)$ (after appropriate scaling) by the method of moments \cite[Theorem 30.2]{billingsley}, which can be easily adapted to our general situation. As we are sure that the result is well-known, we omit the proof.
\begin{trm}[Method of moments]\label{MKM}
Let $Y,Y_1,Y_2,\ldots$ be $\b M_{p,q}$ valued random variables. Suppose that the distribution of $Y$ is determined by its moments $M_k(Y)$ $(k\in \b N)$, that the $Y_n$ have moments $M_k(Y_n)$ of all orders, and that 
\beo  \lim\limits_{n\to\infty}M_k(Y_n)=M_k(Y)\eeo
for $k=1,2,\ldots$. Then the sequence $(Y_n)_n$ converges to $Y$ in distribution.
\end{trm}
\begin{rem}
Each matrix variate normal distribution $\c N(M,\Sigma)$ on $\b M_{p,q}$ or distribution with compact support are determined by its moments.
\end{rem}
\begin{dfn}
Let $(D_n)_{n\in \b N}$, $(d_n)_{n\in \b N}$ be a sequences of matrices from $\b M_q$ and positive real numbers respectively. We write 
$D_n=O(d_n)$ as  $n\to\infty$, if and only if 
$\left\|D_n\right\|_{\infty}=O(d_n)$ as $n\to \infty$. 
\end{dfn}
\begin{pro}\label{A1Convergence}
Assume that $\nu\in \c M^1(\Pi_q)$ has compact support. Then the asymptotic behaviour of $\f A_n:=\f A_{n}(\nu)$ is given as follows:
\begin{itemize}
\item[(a)] If $\frac{n}{p_n}\to c\in [0,\infty[$  as $n\to \infty$, then $\frac{1}{\sqrt{n}} \f A_n$ tends in distribution to  $\c N(\mathbf{0},\Sigma(\nu))$.
\item[(b)] If $\frac{n}{p_n}\to \infty $  as $n\to \infty$, then  $\frac{\sqrt{p_n}}{n} \f A_n$ tends in distribution to $\delta_{\mathbf{0}}$.
\end{itemize}
\end{pro}
\proof
If we prove that for all $k\in \b N_0$, the $k$-th order moments 
\be\label{MomenteA}  \frac{1}{n^{k/2}}E \Bigl( \f A_n^{\otimes, k} \Bigl) \quad \text{and} \quad  \frac{p_n^{k/2}}{n^{k}}E \Bigl( \f A_n^{\otimes, k} \Bigl)\ee
tend to the $k$-th order moment of the corresponding limit distribution in the case (a) and (b) respectively, the assertion follows by the method of moments  \ref{MKM}. Therefore, we calculate (\ref{MomenteA}) as $n\to \infty$. 
Since the random variables $A_j$ are identically distributed, Theorem \ref{TrmMNTkreuz} shows that
\beo  \b E \Bigl( \f A_n^{\otimes, k} \Bigl) = \sum \limits_{u=1}^{k} \sum\limits_{\lambda\in C(k,u)} 
\binom{n}{u}\sum\limits_{\pi\in \mathfrak{S}(\lambda )} \b E \pi\big(A_{1},\dots, A_{u}\big).  \eeo
For $u\in \left\{1,\dots, k\right\}$ and $\lambda\in C(k,u)$ we consider 
\be\label{def_TA} T(\lambda) :=  \binom{n}{u} \sum\limits_{\pi\in \mathfrak{S}(\lambda )} \b E \pi (A_{1},\dots, A_{u}) \in \b M_{q^k}.\ee 

If $\lambda_{\alpha}=1$ for some $\alpha$, i.e $A_{\alpha}$ appears exactly once in $\pi\big(A_{1},\dots, A_{u}\big)$, then each summand in (\ref{def_TA}) vanishes, which is due to the facts that $\b E(A_{\alpha})=\mathbf{0}\in \b M_q$ and that the $A_i$ are independent. 

Suppose that $\lambda_{\alpha}\geq 2$ for each $\alpha$ and $\lambda_{\alpha}>2$ for some $\alpha$. Then $k>2u$, and since $T(\lambda)=O(n^u)$ as $n\to \infty$, 
it follows that $(1/n^{k/2})\cdot T(\lambda)$ and $(p_n^{k/2}/n^{k})\cdot T(\lambda)$ in the cases  $n/p_n \to c\in[0,\infty[$ and $n/p_n \to \infty $ respectively tend to zero as $n\to \infty$. 

Now we turn to the case $\lambda=(2,\dots,2)$, in particular $k=2u$. Let $Z_1,\dots, Z_u$ be independent and $\c N(\mathbf{0},\Sigma^2(\nu))$ distributed random variables. By Lemma \ref{PermutationKroneker}, for any $\pi\in \f S (\lambda)$ there exist permutation matrices $P_\pi$ and $Q_\pi$ with
\begin{align*} P_\pi \b E \bigl( \pi (A_1,\ldots,A_u) \bigl) Q_\pi &= \b E \bigl( A_1\otimes A_1 \otimes \ldots \otimes A_u\otimes A_u \bigl)=\Sigma(\nu)\otimes \ldots \otimes \Sigma(\nu)\\
 &= \b E \bigl( Z_1\otimes Z_1 \otimes \ldots \otimes Z_u\otimes Z_u \bigl) = P_\pi \b E \bigl(\pi(Z_1,\ldots ,Z_u) \bigl) Q_\pi, \end{align*}
and hence
\beo \b E \pi (A_1,\ldots,A_u) = \b E \pi(Z_1,\ldots ,Z_u) \quad \forall \ \pi \in \f S(\lambda).\eeo
Therefore, according to the Lemma \ref{MomenteNormalDistribution} we have
\beo T(\lambda)= \binom{n}{u} \sum\limits_{\pi\in \mathfrak{S}(\lambda )} \b E \bigl( \pi\big(Z_{1},\dots, Z_{u}\big) \bigl)=\frac{n!}{(n-u)!}M_k(Z_1). \eeo 
This proves that the moments in (\ref{MomenteA}) converge to those of $\c N(\mathbf{0},\Sigma^2(\nu))$ and the Dirac distribution $\delta_{\mathbf{0}}$ respectively. 
\qed
\\

Now we introduce some notation: 
Let $k,\ n\in\b N$ and $\c I_{k,n}$ the set of all  $2k$-tuples  $\left(i_1,j_1,\dots, i_k,j_k \right)$ of positive integers less or equal $n$ such that $i_{\alpha}\neq j_{\alpha}$ for all $\alpha=1,\dots,k$. For an $I\in \c I_{k,n}$ and $\pi=(\pi_1,\ldots,\pi_k)\in \b N^k$ we set
\be\label{addend_Ipi} S(I,\pi) := \left[X_{i_1}^{(\pi_1,\alpha_1)}X_{j_1}^{(\pi_1,\beta_1)}\right]_{1\leq \alpha_1,\beta_1\leq q} \otimes \ldots \otimes \left[X_{i_k}^{(\pi_k,\alpha_k)}X_{j_k}^{(\pi_k,\beta_k)} \right]_{1\leq \alpha_k,\beta_k\leq q}. \ee
Each entry of $S(I,\pi) \in \b M_{q^k}$ is a product with $k$ factors and corresponds to the tuple
\be\label{addend_Ipi_tuple} \bigl((i_1, \pi_1,\alpha_1),(j_1,\pi_1,\beta_1),\ldots, (i_k,\pi_k,\alpha_k),(j_k,\pi_k,\beta_k)\bigr). \ee
For (\ref{addend_Ipi_tuple}) and two integers $a$, $b$ we define 
\beo mult_{I,\pi}(a,b)= \left| \left\{ \tau \in  \left\{1,\ldots,k\right\}: (i_{\tau},\pi_{\tau},\alpha_{\tau})=(a,b,\alpha_{\tau}) \text{ or } (j_{\tau},\pi_{\tau},\beta_{\tau})=(a,b,\beta_{\tau}) \right\}  \right|. \eeo
It is clear that $mult_{I,\pi}(a,b)$ does not depend on the indices $\alpha_{\tau}$ and $\beta_{\tau}$. Therefore, $mult_{I,\pi}(a,b)$ is the number of factors in an arbitrary entry of the matrix $S(I,\pi)$ which are coming from the $b$-th row of $X_a$. 
Moreover, we write $d(I)$ for the number of distinct elements in $\left\{I\right\}$. For an $m\in \left\{2,\ldots, 2k\right\}$ and $M\subsetneqq \left\{1,\ldots, n\right\}$ with $\left|M\right|\leq k$ we consider following subsets of $\c I_{k,n}$
\begin{align*}
\c J_m&:=\left\{I\in \c I_{k,n}:\p d(I)=m \right\},\\
\tilde{\c J}_m&:=\left\{I\in \c J_{m}:\p \left\{I\right\} = \left\{1,\ldots, m\right\} \right\},\\
\c J_k^{\exists}(M)&:=\left\{I\in \c J_{2}\cup \ldots \cup \c J_k:\p \left\{I\right\}\cap M \neq \emptyset  \right\},\\
\c J_k^{\forall}(M)&:=\left\{I\in J_{2}\cup \ldots \cup \c J_k:\p m\in \left\{I\right\} \ \forall \ m \ \in \ M  \right\},\\
\c J^{o}(\pi)&:=\left\{I\in \c I_{k,n}:\p \exists \ a,\ b \in \b N: \p mult_{I,\pi}(a,b) \text{ is odd} \right\}.
\end{align*}
It is easily checked that for the cardinalities of $\c J_m$, $\c J_k^{\exists}(M)$ and $J_k^{\forall}(M)$ we have  
\be\label{JmOrder} \left|\c J_m\right|\leq C n^m, \p \left|\c J_k^{\exists}(M)\right|\leq C n^{k-1},\p \left|\c J_k^{\forall}(M)\right|\leq C n^{k-\left|M\right|} \ee with some constant $C=C(k)$.
\begin{pro}\label{B1Convergence}
Assume that $\nu\in \c M^1(\Pi_q)$ has compact support. Then the asymptotic behaviour of
$\f B_n:=\f B_{n}(\nu)$ is given as follows:
\begin{itemize}
\item[(a)] If $\frac{n}{p_n} \to 0$  as $n\to \infty$, then $\frac{1}{\sqrt{n}} \f B_n$ tends in distribution to $\delta_{\mathbf{0}}$.
\item[(b)] If $\frac{n}{p_n}\to c\in ]0,\infty] $  as $n\to \infty$, then  $\left(\frac{\sqrt{p_n}}{n}\right) \f B_n$ tends in distribution to the normal distribution $\c N(\mathbf{0},T(\nu))$.
\end{itemize}
\end{pro}
\proof
According to the Theorem \ref{MKM} it suffices to show that the $k$-th moments of $\tfrac{1}{\sqrt{n}}\f B_n$ and $\tfrac{\sqrt{p_n}}{n} \f B_n$
tend to the corresponding ones of the limiting distributions as $n \to \infty$.
By using very similar arguments as in the proof of the Lemma \ref{CovStructure_nup} it is easily seen that $B_i$ $(i=1,2,\ldots)$ are identically distributed. From this and Theorem \ref{TrmMNTkreuz} we conclude
\beo  \b E \Bigl( \f B_n^{\otimes, k} \Bigl) = \sum \limits_{v=1}^{k} \sum\limits_{\mu\in C(k,v)} \binom{p_n}{v}
\sum\limits_{\pi\in \mathfrak{S}(\mu )} \b E \pi\big(B_{1},\dots, B_{v}\big).  \eeo
For an $v\in \left\{1,\dots,k\right\}$, $\lambda\in C(k,v)$ and $\pi\in \f S(\mu)$ we consider $\pi(\row Bv)$. The definition of $B_{a}$ $(a\in M_v)$ in (\ref{DefB_n}) enables us to write
\be \label{pi(B)} \pi(\row Bv)=B_{\pi_1}\otimes\ldots\otimes B_{\pi_k}=\sum\limits_{I\in \c I_{k,n}} S(I,\pi), \ee
where each term $S(I,\pi)$ with $I=\left(i_1,j_1,\dots, i_k,j_k\right)$ is given by (\ref{addend_Ipi}).
For a selected index $a\in M_n$, each entry of $S(I,\pi)$ may be regarded as a monomial in the variables $X_{a}$ (that is, in $X_{a}^{(\alpha,\beta)}$ with $\alpha$, $\beta \in \b N$) while the random variables coming from other indices are considered as constant. In this view, for any $I\in \c J^{o}(\pi)$, each entry of $S(I,\pi)$ is for some $a\in \left\{1,\ldots,n\right\}$ and $b\in \left\{1,\ldots,v\right\}$ a monomial in the variable $X_{a}$ which is not even in row $b$. And hence Theorem \ref{kMomentof_nu_p} clearly forces 
\be\label{ESIpi=0} \b E(S(I,\pi))=\mathbf{ 0 } \p \forall \ I \in \c J^{o}(\pi).\ee
Therefore, since $\c J_m\subset \c J^{o}(\pi)$ for $m>k$, we conclude from (\ref{pi(B)}) that
\be\label{addend_EIpi}  \b E \pi(\row Bv) =  \sum\limits_{m=2}^{k}\sum_{ \p  I\in \c J_m}  \b E S(I,\pi). \ee
By the definition of $S(I,\pi)$ in (\ref{addend_Ipi}) and Theorem \ref{kMomentof_nu_p}, the terms in the last sum are uniformly bounded by $C\cdot n^m$ with a constant $C>0$, that is,
\beo \sup\limits_{I\in \c J_m} \left\|E S(I,\pi)\right\|_{\infty} = O(n^m). \eeo 
Moreover, according to (\ref{JmOrder}) we have $\left|\c J_m\right|\leq Cn^m$ for a constant $C>0$, and hence we get
\be \label{OrdnungEpiBu}  E \pi(\row Bv) = \sum_{ \p  I\in \c J_k }   \b E S(I,\pi) + O\left(\frac{n^{k-1}}{p_n^k}\right) = O\left(\frac{n^{k}}{p_n^k}\right). \ee
For $v\in \left\{1,\dots, k\right\}$ and $\mu\in C(k,v)$ let us consider 
\be\label{def_T} T(\mu) := \binom{p_n}{v} \sum\limits_{\pi\in \mathfrak{S}(\mu)} \b E  \pi ( B_{1},\dots, B_{v} ).\ee 

If $\mu_{\alpha}=1$ for some $\alpha$, i.e for any $\pi\in \f S(\mu)$ the factor $B_{\alpha}$ 
appears exactly once in the product $\pi\big(B_{1},\dots, B_{v}\big)$, and therefore, each $I\in \c I_{k,n}$ from the Representation (\ref{pi(B)}) of $\pi\big(B_{1},\dots, B_{v}\big)$ is necessarily from $\c J^{o}(\pi)$, 
and hence (\ref{ESIpi=0}) gives  $T(\mu)=0$. 

Suppose that $\mu_{\alpha}\geq 2$ for each $\alpha$ and $\mu_{\alpha}>2$ for some $\alpha$, that is, in particular $k>2v$. From (\ref{OrdnungEpiBu}) we conclude that  $n^{-k/2}T(\mu)=O(n^{k/2}p_n^{v-k})$ and $p_n^{k/2}n^{-k} T(\mu)=O
(p_n^{v-k/2})$
tend to $\mathbf{0}$ as $n\to \infty$ in the case (a) $\frac{n}{p_n}\to 0$ and case (b) $\frac{p_n}{n}\to 0$ respectively.

We now turn to the case $\mu=(2,\dots,2)$, in particular $k=2v$. By Eq. (\ref{OrdnungEpiBu}) it follows in the case (a) that  $n^{-v}T(\mu)=O((n/p_n)^{k-v})$ and hence that $n^{-v}T(\mu)$ converges to zero as $n\to \infty$. 

Since $X_1,X_2,\ldots$ are i.i.d., we have
\beo 
\sum\limits_{I\in \c J_{k}} \b E  S(I,\pi) = \binom{n}{k} \sum\limits_{I\in \tilde{\c J}_{k}}  \b E S(I,\pi). \eeo
Therefore, by using Eq. (\ref{OrdnungEpiBu}),
\beo T(\mu)=\frac{p_n!}{(p_n-v)!}\frac{n!}{(n-k)!}\frac{1}{v!} 
\sum\limits_{\pi\in \mathfrak{S}(\mu)} \frac{1}{p_n^k} \frac{p_n^k}{k!}  \sum\limits_{I\in \tilde{\c J}_{k}}    \b E S(I,\pi) + O\left(\frac{n^{k-1}}{p_n^{k-v}}\right).
\eeo
Let $Z_1,\dots, Z_v$ be independent and $\c N(\mathbf{0},T(\nu))$ distributed random variables. By Lemma \ref{DMO}, which is proven below, we obtain
\begin{align*} \lim\limits_{n\to\infty} \frac{p_n^v}{n^k}T(\mu)= \frac{1}{v!} \sum\limits_{\pi\in \mathfrak{S}(\mu)} \b E \pi (Z_1,\ldots, Z_v).
\end{align*}
The required result then follows from Lemma \ref{MomenteNormalDistribution} and Method of moments \ref{MKM}.
\qed
\begin{lmm}\label{DMO}
Let $v\in \b N$, $k=2v$, $\mu=(2,\ldots,2)\in C(k,v)$, $\pi \in \f S(\mu)$ and $Z_1,\ldots, Z_v$ be independent $\c N(\mathbf{0},T(\nu))$ distributed random variables. Then
\beo \b E \pi (Z_1,\ldots, Z_v)  = \frac{p_n^k}{k!} \sum\limits_{I\in \tilde{\c J}_{k}} \b E S(I,\pi) =:R(\pi).\eeo
\end{lmm}
\proof
According to the Lemma \ref{PermutationKroneker} there is no loss of generality in assuming that $\pi=(1,1,2,2,\ldots, v,v)$. We set
\beo \c J_{k,\pi}:=\left\{(i_1,j_1,\ldots,i_k,j_k)\in \tilde{\c J}_{k}: \quad \left\{i_{\alpha},j_\alpha \right\}= \left\{ i_\beta,j_\beta \right\} \text{ if }\pi_{\alpha}=\pi_{\beta} \right\}. \eeo 
It is easy to check that $\tilde{\c J}_k \setminus \c J_{k, \pi } \subset J^{o}(\pi)$. Therefore, by Eq. (\ref{ESIpi=0}),
\beo  \sum\limits_{I\in \tilde{\c J}_{k}} \b E S(I,\pi) = \sum\limits_{I \in \c J_{k,\pi}}  \b E  S(I,\pi). \eeo
For a permutation $\sigma\in S_k:=Sym(\left\{1,\ldots,k\right\})$ and $\epsilon=(\epsilon_1,\ldots,\epsilon_u) \in \b Z_2^v$ we consider the functions 
\begin{align*} \varphi_{\sigma}:& \c J_{k,\pi} \longrightarrow  \c J_{k,\pi}, \quad &(i_1,j_1,\ldots,i_k,j_k)& \mapsto (\sigma(i_1),\sigma(j_1),\ldots, \sigma(i_k),\sigma(j_k))
\\ \theta_{\epsilon}:& \c J_{k,\pi} \longrightarrow \c J_{k,\pi}, \quad &(i_1,j_1,\ldots,i_k,j_k)& \mapsto (r_1,t_1,\ldots,r_k,t_k),
\end{align*} 
where $(r_1,t_1,\ldots,r_k,t_k)$ is defined as follows: for any $\alpha, \beta\in M_k$ with $\alpha<\beta$ and $\pi_{\alpha}=\pi_{\beta} \in \left\{1,\ldots,v\right\}$ we have
\beo (r_{\alpha},t_{\alpha},r_{\beta},t_{\beta})=  \begin{cases}  (i_{\alpha},j_{\alpha},i_{\beta},j_{\beta}),  & \text{if }\epsilon_{\pi_{\alpha}}=0, \\ (i_{\alpha},j_{\alpha},j_{\beta},i_{\beta}), & \text{if }\epsilon_{\pi_{\alpha}}=1.\end{cases} \eeo 
It is easily seen that $\varphi_{\sigma}$ and $\theta_{\epsilon}$ are well defined. Let $I_0:=(1,2,1,2,\ldots, k-1,k,k-1,k)\in \c J_{k,\pi}$. By standard verification we obtain a one-to-one correspondence between  $S_k \times \b Z _2^v$ and $\c J_{k,\pi}$ via the map $\Psi:(\sigma,\epsilon)\mapsto \varphi_{\sigma}\bigl(\theta_{\epsilon}(I_0)\bigl)$. Since $X_1,X_2,\ldots $ are independent identically distributed we have for all $\sigma\in S_k$
\be\label{PermutationInvariant} \b E S( \varphi_\sigma (I), \pi )= E S(I, \pi ) \quad \forall \ I \in \c J_{k,\pi}.  \ee
For an $\epsilon \in \b Z_2$ we consider the algebraic operation
\beo \epsilon(a,b)=  \begin{cases}  a,  & \text{if }\epsilon=0, \\ b, & \text{if }\epsilon=1.\end{cases}  \eeo
By Equation (\ref{PermutationInvariant}) it follows that 
\begin{align*}
R(\pi) & =  \frac{1}{k!}\sum\limits_{(\sigma,\epsilon) \in S_k \times \b Z_2^v} p_n^k \b E S(\Psi(\sigma,\epsilon),\pi) =  \sum\limits_{\epsilon \in \b Z_2^v} p_n^k \b E  S(\Psi(id,\epsilon),\pi)  \\
& = \sum\limits_{\epsilon \in \b Z_2^v} \epsilon_1(T_1,T_2)\otimes \ldots \otimes \epsilon_v(T_1,T_2)
= \b E  \pi (Z_1,\ldots, Z_v),
\end{align*}
where $T_1:=T_1(\nu)$ and $T_2:=T_2(\nu)$ are defined as in (\ref{T1T2}).
\qed

Now, in order to prove Theorem \ref{mainResult} for sequences $p_n$ with $p_n/n\to c\in ]0,\infty[$ we show that $\f A_n:=\f A_n(\nu)$ and $\f B_n:=\f B_n(\nu)$ are asymptotically independent. 
\begin{pro}\label{asympUncorrelated}
Assume that $\nu\in \c M^1(\Pi_q)$ has compact support and that $\lim\limits_{n\to \infty}\frac{n}{p_n}=:c \in ]0,\infty[$.  Then the random variables $\f A_{n}$ and $\f B_{n}$ are asymptotically independent, that is, for all $0\leq l\leq k$ and all $\sigma \in \f S( (l,k-l))$ 
\beo F(n;\sigma):=\frac{1}{n^{k/2}} \left[ \b E \sigma( \f A_{n} ,  \mathbbmtt{1} ) \circ \b E \sigma(  \mathbbmtt{1}, \f B_{n} )- \b E \sigma(  \f A_{n}, \f B_{n} ) \right]  \eeo
tends to zero as $n\to \infty$.
\end{pro}
\proof
According to the Lemma \ref{PermutationKroneker} there is no loss of generality in assuming that $\sigma=(1,\ldots 1,2,\ldots ,2)$. From Theorem \ref{TrmMNTkreuz}, by using symmetry argument, we conclude
\begin{align*} F(n;\sigma)&=\frac{1}{n^{k/2}} \Bigl( \b E \Bigl( \f A_n^{\otimes, l} \Bigl)\otimes \b E \Bigl( \f B_n^{\otimes, k-l} \Bigl)-  \b E \Bigl( \f A_n^{\otimes, l} \otimes  \f B_n^{\otimes, k-l} \Bigl) \Bigl) \\ &= \frac{1}{n^{k/2}} \sum \limits_{u=1}^{l}\sum \limits_{v=1}^{k-l} \sum\limits_{\lambda\in C(l,u)} \sum\limits_{\mu \in C(k-l,v)} \binom{n}{u}\binom{p_n}{v}
\sum\limits_{\pi\in \mathfrak{S}(\lambda )}\sum\limits_{ \pi' \in \mathfrak{S}(\mu )} H(\pi,\pi'), 
\end{align*}
with
\beo H(\pi,\pi')=\b E \pi (A_{1},\dots, A_{u}) \otimes \b E \pi' (B_{1},\dots, B_{v}) - \b E \bigl( \pi(A_{1},\dots, A_{u}) \otimes \pi' (B_{1},\dots, B_{v}) \bigl).
\eeo

If $\mu_{\alpha}=1$ for some $\alpha\in \left\{1,\ldots, v\right\}$, then each entry of $\pi'(B_1,\ldots, B_v)$ is not an even polynomial and thus so is $\pi(A_1,\ldots,A_u)\otimes \pi' (B_1,\ldots, B_v)$ neither. Therefore, $H(\pi,\pi')=0$ by Theorem \ref{kMomentof_nu_p}. 

Suppose that $\mu_{\alpha}\geq 2$ for each $\alpha$. By Eq. (\ref{addend_EIpi}) we have
\be\label{CpivarpiSummands} H(\pi,\pi')=\sum_{I\in \c J_2\cup \ldots \cup \c J_{k-l}} 
\left( \b E \pi\big(A_{1},\dots, A_{u}\big) \otimes \b E  S(I,\pi') - \b E \bigl( \pi (A_{1},\dots, A_{u} ) \otimes S(I,\pi') \bigr) \right). \ee
Let $M:=\left\{1,\ldots, u\right\}$ and $G:=\left\{\alpha\in M:\ \lambda_\alpha=1 \right\}$. We consider the $I$-th term in the sum above, which will be denoted by $T(I)$. Is $I\notin \c J_{k-l}^{\exists}(M)$, that is, $\left\{I\right\}\cap M= \emptyset$, and thus $A_1,\ldots, A_u$ are independent  from $S(I,\pi')$. This clearly forces $T(I)=0$. Is $I\notin \c J_{k-l}^{\forall}(G)$, that is, there exists $\tau \in G$ with $\tau \notin \left\{I\right\}$, and therefore, $A_\tau$ is independent from $A_i$ $(i\in M \setminus \left\{\tau \right\})$ and $S(I,\pi')$.  We thus get $T(I)=0$ from (\ref{A_structure}). 

Taking (\ref{JmOrder}) into account, we see that the number of nonzero summands in (\ref{CpivarpiSummands}) is bounded above $\min(n^{k-l-1},n^{k-l-\left|G\right|})$. On the other side, Lemma \ref{kMomentof_nu_p} yields that each of them is bounded above $C/p_n^{k-l}$ where $C>0$ is a suitable global constant. 
Summarized we get 
\be\label{EstimationC_pi}\left\|H(\pi,\pi')\right\| \leq C \cdot \min(n^{-1},n^{-\left|G\right|}).\ee
Since $\mu\in C(k-l,v)$ with $\mu_{\alpha}\geq 2$ for all $\alpha\in \left\{1,\cdots,v\right\}$ we have that $k-l\geq 2v$. Moreover, since $\lambda\in C(l,u)$ we get $l\geq 2u-\left|G\right|$. And hence, by straightforward calculation using $n/p_n\to c \in ]0,\infty[$ we conclude from (\ref{EstimationC_pi}) that for suitable constants $C_i$,
\begin{align*} \left\|F(n,\sigma)\right\|& \leq \frac{C_1}{n^{k/2}} \sum \limits_{u=1}^{l} \sum \limits_{v=1}^{k-l} \sum\limits_{\lambda\in C(l,u)} \sum\limits_{\mu \in C(k-l,v)} \binom{n}{u}\binom{p_n}{v} \min(n^{-1},n^{-\left|G\right|}) \\
& \leq \frac{C_2}{n^{k/2}}
\sum \limits_{u=1}^{l} \sum \limits_{v=1}^{k-l} \sum\limits_{\lambda\in C(l,u)}  n^{u+v} \min(n^{-1},n^{-\left|G\right|}) \leq  \frac{C_3} {\sqrt{n}}.
\end{align*}
This completes the proof.
\qed \\

\begin{proof}[Proof of Theorem \ref{mainResult} for $\nu\in \c M^1(\Pi_q)$ with compact support.] \hspace{0pt} \\
If $n/p_n\to \infty$ then $\frac{\sqrt{p_n}}{n}\f A_n \stackrel{d}{\rightarrow} \delta_{\mathbf{0}}$ and $\frac{\sqrt{p_n}}{n}\f B_n \stackrel{d}{\rightarrow} \c N(\mathbf{0},T(\nu))$ according to Propositions \ref{A1Convergence} and \ref{B1Convergence}. This clearly forces $\frac{\sqrt{p_n}}{n} \Xi_n(\nu) \stackrel{d}{\rightarrow} \c N(\mathbf{0},T(\nu))$ by Slutsky's Theorem. 
Suppose that $n/p_n\to 0$. Then we get as above $\frac{1}{\sqrt{n}} \Xi_n(\nu) \stackrel{d}{\rightarrow} \c N(\mathbf{0},\Sigma(\nu))$. It remains only to check the convergence in the case $n/p_n\to c\in ]0,\infty[$. Let $k\in \b N$. By Theorem \ref{TrmMNTkreuz}, 
\begin{align*}
M_{k}(\Xi_n(\nu))=\b E \left((\f A_n+\f B_n)^{\otimes k}\right)  =  \sum\limits_{l=0}^{k} \sum\limits_{\pi\in \mathfrak{S}((l,k-l) )} \b E \pi(\f A_n,\f B_n).
\end{align*}
Therefore, by Proposition \ref{asympUncorrelated},
\begin{align*}
\lim\limits_{n\to \infty}M_{k}\Bigl(\frac{1}{\sqrt{n}}\Xi_n(\nu)\Bigl)=\lim\limits_{n\to \infty} \frac{1}{n^{k/2}} \sum\limits_{l=0}^{k} \sum\limits_{\pi\in \mathfrak{S}((l,k-l) )} \b E \pi(\f A_n,\mathbbmtt{1}) \circ \b E \pi(\mathbbmtt{1},\f B_n).
\end{align*}
Consider independent random variables $Z_1$, $Z_2$ and $Z$ with distributions $\c N(\mathbf{0},\Sigma(\nu))$, $\c N(\mathbf{0},cT(\nu))$ and $\c N(\mathbf{0},\Sigma(\nu)+cT(\nu))$ respectively. Propositions $\ref{A1Convergence}$, $\ref{B1Convergence}$ and Lemma $\ref{Z1Z2ZNormaldistributions}$ now lead to 
\begin{align*}
\lim\limits_{n\to \infty}M_{k}(\Xi_n(\nu))= \sum\limits_{l=0}^{k} \sum\limits_{\pi\in \mathfrak{S}((l,k-l) )} \b E  \pi(Z_1,\mathbbmtt{1}) \circ \b E \pi(\mathbbmtt{1}, Z_2) 
=M_k(Z).
\end{align*}
\end{proof}

In order to get rid of the assumption that $ supp(\nu)$ is compact, we introduce for an $a>0$ the truncated $\b M_{p_n,q}$-valued random variables 
\beo X_{k,a}:=\begin{cases}  X_k,  & \text{if }\left\| \varphi_{p_n}(X_{k}) \right\| \leq a, \\ 
\mathbf{ 0 }, & \text{otherwise }\end{cases} \p k=1,2,\ldots \eeo
Let us denote by $\nu_a$ the distribution of $\varphi_{p_n}(X_{1,a})$ (which is not dependent on $p_n$). Obviously, the sequence $X_{k,a}$, $k\in \b N$, are i.i.d. with the radial law $\nu_{p_n,a}\in \c M(\b M_{p_n,q})$ which corresponds to $\nu_a$. We define $\Xi_n(\nu_a)$, $\f A_n(\nu_a)$, $A_{j,a}$ $(j=1,\ldots,n)$, $\f B_n(\nu_a)$ and $B_{j,a}$ $(j=1,\ldots,p_n)$ according to (\ref{mainProcess}), (\ref{DefA_n}) and (\ref{DefB_n}) respectively, by taking $X_{k,a}$ instead of $X_k$, $k\in \b N$. Clearly, we have $ \Xi_n(\nu_a)= \f A_n(\nu_a)+\f B_n(\nu_a)$.

In the following we show that $\Xi_n(\nu_a)$ is a ''good'' approximation of $\Xi_n(\nu)$.  
To formulate this exactly, we first fix some $\delta>0$ and a sequence $(p_n)_n$; we then introduce the sequence $(\delta_n)_n$ by
\be\label{delta_n} \delta_n:=\begin{cases}
  \delta\cdot \sqrt{n},  & \text{if }\tfrac{n}{p_n}\to c\in[0,\infty[,\\
  \delta\cdot \tfrac{n}{\sqrt{p_n}}, & \text{if }\tfrac{n}{p_n}\to \infty.
\end{cases}\ee
In the next lemmas we show that the events \beo \left\{ \left\| \f A_n(\nu_a)-\f A_n(\nu) \right\| > \delta_n \right \}  \text{ and } \left\{\left\| \f B_n(\nu_a)-\f B_n(\nu) \right\| >\delta_n \right\} \eeo have arbitrary small probabilities for an $a$ and $n$ large enough. \begin{lmm}\label{AnaAnEstimation}
For all $\epsilon>0$, $\delta>0$ there exist $a_0,\ n_0\in \b N$ such that for all $n, \ a\in \b N$ with $a\geq a_0$ and $n\geq n_0$ 
\begin{align*}    
\b P \left( \left\| \f A_n(\nu)-\f A_{n}(\nu_a) \right\| > \delta_n \right) \leq \epsilon.
\end{align*}

\end{lmm}
\proof
Let  $\delta>0$ and $(\delta_n)_n$ be a sequence as in (\ref{delta_n}). Since $(A_i-A_{i,a})$, $(i=1,2,\ldots$) are i.i.d., it follows by Chebychev inequality that
\be\label{EstimationPAaA} \b P\left( \left\|\f A_n(\nu)-\f A_{n}(\nu_a)\right\|\geq \delta_n\right)\leq \frac{n}{\delta_n ^2}\b E\left( \left\| A_1-A_{1,a} \right\|^2\right). \ee
Using triangle inequality we obtain
\beo \sup\limits_{a\in\b N }\left\|A_{1,a}\right\|^2\leq \left(\left\|\varphi_{p_n}^2(X_{1})\right\|+\left\|r_2(\nu)\right\|\right)^2 \in L^1(\Omega), \eeo
Therefore, the set $\{ \|A_{1,a} \|^2: \ a\in \b N \}$ is uniformly integrable. On the other side, since the random variable $\left\|A_1\right\|$ is almost surely finite, $\left\|A_{1,a}\right\|^2$ converges a.s. to $\left\|A_{1}\right\|^2$ as $a\to \infty$. We thus get
\be\label{L1ConvergenceAa} \left\|A_{1,a}\right\|^2 \longrightarrow \left\|A_{1}\right\|^2 \p \text{in }L^1.\ee
By taking (\ref{EstimationPAaA}) 
and (\ref{L1ConvergenceAa}) into account, the lemma follows. 
\qed
\begin{lmm}\label{BnaBnEstimation}
For all $\epsilon>0$, $\delta>0$ there exist $a_0,\ n_0\in \b N$ such that for all $n, \ a\in \b N$ with $a\geq a_0$ and $n\geq n_0$ 
\begin{align*}
\b P \left( \left\| \f B_n(\nu)-\f B_{n}(\nu_a) \right\| > \delta_n \right) \leq \epsilon.
\end{align*}
\end{lmm}
\proof
Let $\delta>0$ and $(\delta_n)_n$ be a sequence as in (\ref{delta_n}). By Chebychev inequality it follows that 
\be\label{EstimationPBaB} \b P\left( \left\|\f B_n(\nu)-\f B_{n}(\nu_a)\right\|\geq \delta_n\right) \leq \frac{1}{\delta_n ^2}\sum\limits_{j,i=1}^{p_n} \b E\left( \left\langle  B_i-B_{i,a} , B_j-B_{j,a} \right\rangle   \right). \ee
Using Lemma \ref{CovStructure_nup} one can easily compute that 
\begin{align*}
\b E \left( \left\langle B_i,B_j\right\rangle \right) &= \delta_{ij}\cdot \frac{n(n-1)}{p_n^2}\sum\limits_{l,k=1}^{q}r_2(\nu)_{l,l}r_2(\nu)_{k,k}+r_2(\nu)_{l,k}r_2(\nu)_{l,k} \\
\b E \left( \left\langle B_{i,a},B_{j,a}\right\rangle \right) &= \delta_{ij}\cdot \frac{n(n-1)}{p_n^2}\sum\limits_{l,k=1}^{q}r_2(\nu_a)_{l,l}r_2(\nu_a)_{k,k}+r_2(\nu_a)_{l,k}r_2(\nu_a)_{l,k} 
\end{align*}
With the notation
\beo \tilde{r}_2(a;n):= \left(\b E \left(X_{1,a}^{(1,l)}X_{1}^{(1,k)}\right)\right)_{1\leq l,k\leq q}\eeo
we see at once that
\beo  E \left( \left\langle B_i,B_{j,a}\right\rangle \right) = \delta_{ij} n(n-1)\sum\limits_{l,k=1}^{q}\tilde{r}_2(a;n)_{l,l}\tilde{r}_2(a;n)_{k,k}+\tilde{r}_2(a;n)_{l,k}\tilde{r}_2(n;a)_{l,k}. \eeo
For $l,\ k\in \left\{1,\ldots, q \right\}$ we obtain
\be\label{ranExpression} \tilde{r}_2(a;n)_{l,k}=\frac{1}{p_n} r_2(\nu)_{l,k} - \int_{\left\{\left\|X_1\right\|>a\right\}} X_1^{(1,l)}X_1^{(1,k)} d\b P. \ee
By Cauchy-Schwarz inequality and straightforward calculation we get
\beo 0\leq \Bigl| \int_{\left\{\left\|X_1\right\|>a\right\}} X_1^{(1,l)}X_1^{(1,k)} d\b P \Bigl| \leq \frac{c}{a p_n},\p a\to \infty \eeo uniformly in $n$ with some constant $c>0$.
From this and (\ref{ranExpression}) we deduce 
\beo p_n\tilde{r}_2(a;n)=r_2(\nu)+O\left(\frac{1}{a} \right)  \eeo
and hence 
\beo \forall \ \epsilon >0 \ \exists \ M>0 \ \forall \ n\geq M, \ \forall \ a\geq M: \p 0 \leq \frac{p_n^2}{n^2} \b E \left(\left\|B_i-B_{i,a} \right\|\right) \leq \epsilon. \eeo
Finally, this and (\ref{EstimationPBaB}) lead to the claim.
\qed
\begin{cor}\label{XinaXinEstimation}
For all $\epsilon>0$, $\delta>0$ there exist $a_0,\ n_0\in \b N$ such that for all $n, \ a\in \b N$ with $a\geq a_0$ and $n\geq n_0$ 
\begin{align*}
\b P \left( \left\|  \Xi_n(\nu)-\Xi_{n}(\nu_a) \right\| > \delta_n \right) \leq \epsilon,
\end{align*}
where $\delta_n=\delta \sqrt{n}$ if $n/p_n\to c \in [0,\infty[$ and $\delta_n=\delta \frac{n}{\sqrt{p_n}}$ if $n/p_n\to \infty$.
\end{cor}
\proof
For an $\delta>0$ we observe
\beo \b P \left( \left\|  \Xi_n(\nu)-\Xi_{n}(\nu_a) \right\| > \delta_n \right)\leq \b P \bigl( \left\|  \f A_n-\f A_{n,a} \right\| > \frac{\delta_n}{2} \bigl)+\b P \bigl( \left\|  \f B_n-\f B_{n,a} \right\| > \frac{\delta_n}{2} \bigl). \eeo
Combining this with Lemmas \ref{AnaAnEstimation} and \ref{BnaBnEstimation}, the corollary follows. 
\qed

\begin{proof}[Proof of Theorem \ref{mainResult}] \hspace{0pt} \\
Let us first prove the CLT I. In this case the normalisation  is given by $\frac{\sqrt{p_n}}{n}$ and for the growth of $p_n$ we have the condition $n/p_n\to \infty$ as $n\to \infty$. We set $\xi_n:=\frac{\sqrt{p_n}}{n}\Xi_n(\nu)$ and $\xi_{n,a}=\frac{\sqrt{p_n}}{n}\Xi_n(\nu_a)$ and denote their distributions by $\mu_n$ and $\mu_{n,a}$ respectively. Moreover, we write $\tau_\nu$ instead of $\c N(\mathbf{0},T(\nu))$. Using triangle inequality, we deduce that 
\begin{align}\label{triangleIE} \Big| \int f d \mu_n-\int f d \tau_{\nu} \Big| &\leq 
\Big| \int f  d \mu_{n}-\int f  d \mu_{n,a} \Big|+\\ \notag &+\Big| \int f d \mu_{n,a}-\int f d \tau_{\nu_a} \Big|+
\Big| \int f d \tau_{\nu_a} - \int f   d\tau_{\nu} \Big|.
\end{align}  
Let $\epsilon>0$, $f\in \c C_b^u(\Pi_q)$ be a bounded uniformly continuous function on $\Pi_q$ and $A_{\delta}:=\left\{ \left\| \xi_{n}-\xi_{n,a} \right\| \leq \delta \right\}$ ($\delta>0$). It follows that
\beo \exists \ \delta>0 :\p \int_{A_{ \delta }}\left| f\circ \xi_{n}-f\circ \xi_{n,a} \right| d\b P \leq \epsilon. \eeo
On the other hand, by Corollary \ref{XinaXinEstimation},
\beo \exists \ a_0,\ n_0>0 :\p \int_{\Omega\setminus A_{ \delta }}\left| f\circ \xi_{n}-f\circ \xi_{n,a} \right| d\b P \leq 2\epsilon \left\|f\right\|_{\infty}  \p \forall \ a\geq a_0,\ n\geq n_0. \eeo
This gives us the following estimation for the first summand in (\ref{triangleIE}):
\be\label{summand1}\exists \ a_0,\ n_0>0 :\ \Big| \int f  d \mu_{n}-\int f  d \mu_{n,a} \Big|\leq \epsilon(1+2\left\|f\right\|_\infty) \p \forall \ a\geq a_0,\ n\geq n_0. \ee
Since $\nu_a$ has a compact support, we conclude from \ref{mainResult} that $\mu_{n,a}$ weakly converges to $\tau_{\nu_a}$ $(a>0)$,  hence that
\be\label{summand2} \forall \ a>0 \ \exists  \ n_0>0 :\ \Big| \int f  d \mu_{n,a}-\int f  d \tau_{\nu_a} \Big| \leq \epsilon \p \forall \ n \geq n_0. \ee
Finally, it is evident that
\be\label{summand3}  \exists  \ a_0>0 :\ \Big| \int f d \tau_{\nu_a} - \int f   d\tau_{\nu} \Big| \leq \epsilon \p \forall \ a \geq a_0. \ee
Taking (\ref{summand1}), (\ref{summand2}) and (\ref{summand3}) into account, we obtain
\beo \exists  \ n_0>0 :\ \Big| \int f d \mu_n-\int f d \tau_{\nu} \Big| \leq \epsilon(3+2\left\|f\right\|_\infty) \p \forall \ n \geq n_0, \eeo
which completes the proof of CLT I in Theorem \ref{mainResult}. The same proof works for CLT II. 
\end{proof}

\end{document}